\documentclass[11pt]{amsart}
\usepackage[T1]{fontenc}
\usepackage[latin1]{inputenc}
\usepackage{typearea}
\usepackage{geometry}
\usepackage{ulem}
\usepackage{amsmath}
\usepackage{amssymb}
\usepackage{latexsym}
\usepackage{enumerate}
\usepackage{amsthm}
\usepackage[all]{xy}
\usepackage{hhline}
\usepackage{epsf} 
\usepackage{cite}
\usepackage{verbatim}
\usepackage{mathtools}
\usepackage{comment}
\usepackage{dsfont}
\usepackage{mathrsfs}
\usepackage{hyperref}

\usepackage{color}

\newtheorem{theorem}{{\sc Theorem}}[section]
\newtheorem{cor}[theorem]{{\sc Corollary}}

\newtheorem{lemma}[theorem]{{\sc Lemma}}
\newtheorem{prop}[theorem]{{\sc Proposition}}

\theoremstyle{remark}
\newtheorem{remark}[theorem]{{\sc Remark}}

\theoremstyle{definition}

\newcommand{\R}{\mathbb{R} }
\newcommand{\N}{\mathbb{N} }
\newcommand{\X}{\mathbb{X}}

\newcommand{\A}{\mathcal{A}}
\newcommand{\B}{\mathcal{B}}
\newcommand{\F}{\mathcal{F}}
\newcommand{\G}{\mathcal{G}}

\newcommand{\W}{\mathcal{W}}
\newcommand{\K}{\mathcal{K}}
\newcommand{\Z}{\mathbb{Z}}

\newcommand{\im}{\textnormal{im}}

\newcommand{\Prob}{\mathbb{P}}
\newcommand{\E}{\mathbb{E}}

\newcommand{\Pot}{\mathcal{P}}
\newcommand{\Om}{\Omega}

\providecommand{\abs}[1]{\lvert #1\rvert}
\providecommand{\babs}[1]{\bigl\lvert #1\bigr\rvert}
\providecommand{\Babs}[1]{\Bigl\lvert #1\Bigr\rvert}
\providecommand{\Bbabs}[1]{\biggl\lvert #1\biggr\rvert}

\providecommand{\norm}[1]{\lVert #1\rVert}

\providecommand{\Enorm}[1]{\lVert #1\rVert_2}
\providecommand{\BEnorm}[1]{\Bigl\lVert #1\Bigr\rVert_2}

\DeclareMathOperator{\Var}{Var}

\DeclareMathOperator{\dom}{dom}
\DeclareMathOperator{\Cov}{Cov}

\DeclareMathOperator{\Lip}{Lip}

\DeclareMathOperator{\Id}{Id}

\DeclareMathOperator{\Inf}{Inf}

\renewcommand{\phi}{\varphi}
\renewcommand{\epsilon}{\varepsilon}
\newcommand{\eps}{\varepsilon}
\renewcommand{\rho}{\varrho}
\renewcommand{\P}{\Prob}
\renewcommand{\1}{\mathds{1}}

\begin{document}
\title[Malliavin calculus on product spaces]{On the Malliavin calculus on product spaces and an infinite de Jong theorem }
\author{Christian D\"obler}
\thanks{\noindent Mathematisches Institut der Heinrich Heine Universit\"{a}t D\"usseldorf\\
Email: christian.doebler@hhu.de }
\begin{abstract}  
We extend the Malliavin theory for $L^2$-functionals on product probability spaces that has recently been developed independently by Decreusefond and Halconruy (2019) and by Duerinckx (2021), by  characterizing the domains and investigating the actions of the three Malliavin operators in terms of the infinite Hoeffding decomposition in $L^2$, which we identify as the natural analogue of the famous Wiener-It\^{o} chaos decomposition on Gaussian and Poisson spaces. We further explore the corresponding Ornstein-Uhlenbeck semigroup in terms of its Mehler representation and prove new moment bounds for iterated gradients. As an illustration of the abstract framework, we prove an infinite version of the quantitative de Jong CLT that has recently been proved by G. Peccati and the author (2017) and by the author (2024). 

\end{abstract}

\maketitle

\section{Introduction}\label{intro}
\subsection{Motivation}
In the recent article \cite{DH} by Decreusefond and Halconruy, a version of Malliavin calculus for functionals on a product of countably many probability spaces has been developed. In particular, suitable versions of the \textit{Malliavin derivative} or \textit{gradient} $D$, its adjoint $\delta$, named \textit{divergence operator}, and the corresponding \textit{Ornstein-Uhlenbeck generator} or \textit{number operator} $L$ satisfying the crucial identity $L=-\delta D$ have been introduced. Important properties of the operators $D$ and $L$ have also been established independently by Duerinckx \cite{Duer}.

As has been remarked in \cite{DH}, nowadays, one main motivation for developing such a Malliavin structure on product spaces is to combine it with Stein's method of distributional approximation to derive suitable so-called \textit{Malliavin-Stein bounds} and, consequently, versions of such bounds in the context of normal and gamma approximation have been given in \cite{DH,Duer}. Historically, such Malliavin-Stein bounds have first been proved for functionals of Gaussian processes in the seminal paper \cite{NouPec09} by Nourdin and Peccati and, then, afterwards also for Poisson functionals \cite{PSTU, Schu16} and for functionals of a Rademacher sequence \cite{NPR, KRT1, KRT2}. We refer to the monographs \cite{NouPecbook} and \cite{RePe-book} for comprehensive introductions to the \textit{Malliavin-Stein method} on Gaussian and Poisson spaces, respectively. 
We refer to the constantly updated website \cite{Nweb} for a comprehensive list of research articles related to this line of research.

\subsection{Results and contributions of this work}
Despite the important foundations laid in \cite{DH} and \cite{Duer}, the theory of Malliavin calculus on product spaces still suffers from some shortcomings compared to the well-established theory in the corresponding  Gaussian \cite{Nualart}, Poisson \cite{Lastsa} and Rademacher \cite{Privault} situations. Most importantly, in \cite{DH} no general counterpart to the Wiener-It\^{o} chaos decomposition on those spaces is considered and, consequently,  
explicit descriptions of the domains of the operators $D,\delta$ and $L$ are missing. 
In particular, as in the well-studied situations mentioned above, it would be desirable to understand the actions of these three operators in terms of such a chaotic decomposition. The main purpose of the present paper is to further extend and complement the Malliavin theory for product probability spaces in order to fill these gaps. In particular, we provide a precise description of the domains of the three Malliavin operators $D,\delta$ and $L$. To this end, as a first step, we explicitly identify the infinite Hoeffding decomposition of an $L^2$ functional on a product space as the correct counterpart to the well-known Wiener-It\^{o} chaos decompositions in Gaussian, Poisson and Rademacher situations and study the actions of these three operators in terms of this decomposition. Moreover, we further explore the Ornstein-Uhlenbeck semigroup associated to $L$ via its chaotic and Mehler representations, thus providing analogues of important results in Gaussian and Poisson stochastic analysis.

More precisely, compared to the articles \cite{DH, Duer}, the main new theoretical achievements of this work include
\begin{enumerate}[1)]
\item new characterizations of the domains of the three Malliavin operators $D,\delta$ and $L$ in terms of the infinite Hoeffding decomposition reviewed in Subsection \ref{Hoeffding} and new formulae for the action of the operators $D$ and $L$ in terms of this decomposition,
\item a new Stroock type formula to recover the infinite Hoeffding decomposition of a functional by computing iterated derivatives,
\item new functional analytic properties of the Ornstein-Uhlenbeck operator $L$ and its pseudo-inverse $L^{-1}$,
\item an investigation of the Ornstein-Uhlenbeck semigroup associated to $L$, the corresponding Mehler formula and new moment bounds for iterated gradients composed with the pseudo-inverse $L^{-1}$ of $L$,
\item the introduction of a carr\'{e}-du-champ operator on general product spaces including an alternative representation that allows effectively to assess its non-diffusiveness and the proof of a corresponding integration-by-parts formula.
\end{enumerate} 

We remark that the infinite Hoeffding decomposition has implicitly been exploited in the proof of \cite[Lemma 2.4]{Duer} in order to derive crucial properties of $L$ and $L^{-1}$ without explicitly viewing it as the chaotic decomposition of functionals on product spaces.  \smallskip\\

From a pragmatic point of view, one further motivation for this article is to provide a reference for fundamental theoretical results that may be invoked in other more applied works on normal (and other distributional) approximations. Thus, in the follow-up works \cite{CuDoe26,Doe25c} the abstract theory developed in this paper and in the works \cite{DH, Duer} is combined with Stein's method in order to prove new abstract error bounds for the normal approximation of real-valued functionals defined on such spaces and its usefulness is supported by several relevant applications. 

As an illustration of the abstract theory, we further show here how the general framework may be exploited to (state and) prove an infinite version of the quantitative de Jong CLT that has recently been proved by G. Peccati and the author \cite{DP17} and by the author \cite{Doe23b}. We mention here that even the statement of such a result could not be given without using the concept of infinite Hoeffding decompositions.

Most of the material in the present work is taken from the unpublished manuscript \cite{Doe24} that I decided to split into several parts in order to keep a clearer focus in each of the individual papers. The present paper includes most results concerning the theoretical framework, whereas the works \cite{CuDoe26,Doe25c} (in preparation) deal with the derivation of normal approximation bounds and suitable applications in a Clark-Ocone and Carr\'{e}-du-champ context, respectively. 

The remainder of this article is structured as follows. In Section \ref{malliavin} our new theoretical findings concerning the Malliavin framework on product spaces are presented and proved. In Section \ref{dejong} we state and prove our novel infinite quantitative de Jong type CLT and, finally, in Section \ref{appendix} we give a probabilistic proof of the infinite Hoeffding decomposition in Proposition \ref{infhoeff} below.

\section{Malliavin calculus on product spaces}\label{malliavin}
In this section we review and extend the abstract Malliavin formalism for the product of countably many probability spaces  from \cite{DH}. Contrary to \cite{DH}, for ease of notation, we will assume throughout that the countable index set is given by the natural numbers $\N=\{1,2,\ldots\}$. The case of a general countably infinite index set $A$ may be reduced to this situation by choosing a fixed bijection $\phi:A\rightarrow\N$. Moreover, the situation of a product of only finitely many probability spaces is naturally included by adding countably many copies of a (trivial) one-point probability space. In the latter case, all infinite series appearing in the sequel further reduce to finite sums and the generally subtle problem of identifying the correct domain for the respective operators becomes trivial. 

\subsection{Setup and notation}\label{setup}
We begin by introducing the necessary notation. Thus, let $(E_n,\B_n,\mu_n)$, $n\in\N$, be an arbitrary sequence of probability spaces and denote by 
\[(E,\B,\mu):=\Bigl(\prod_{n\in\N}E_n,\bigotimes_{n\in\N}\B_n,\bigotimes_{n\in\N}\mu_n\Bigr)\]
the corresponding product space. Further suppose that, on a suitable abstract probability space $(\Om,\A,\P)$, 
\[\X=(X_n)_{n\in\N}:(\Om,\A)\rightarrow(E,\B)\quad\text{and}\quad\X'= (X_n')_{n\in\N}:(\Om,\A)\rightarrow(E,\B)\]
are two independent sequences, each distributed according to $\mu$. In particular, the $X_n$, $n\in\N$, are independent and $X_n$ has distribution $\mu_n$.
 Then, for each $j\in\N$, we construct the sequence $\X^{(j)}:=(X_n^{(j)})_{n\in\N}$ by 
\begin{equation*}
 X_n^{(j)}:=\begin{cases}
             X_j',&n=j\\
             X_n, &n\not=j,
            \end{cases}
\end{equation*}
that is we replace the $j$-th coordinate of $\X$ with that of $\X'$. Then, of course, each sequence $\X^{(j)}$ has again distribution $\mu$. In fact, the sequences $\X$ and $\X^{(j)}$ are even \textit{exchangeable}, that is, $(\X,\X^{(j)})$ has the same distribution as $(\X^{(j)},\X)$ for any $j\in\N$. With this construction at hand, we further let $\F:=\sigma(\X)$ and for $M\subseteq\N$ we let $\F_M:=\sigma(X_j,j\in M)$. For $n\in\N$ and $M=[n]:=\{1,\dotsc,n\}$ we write $\F_n:=\F_{[n]}=\sigma(X_1,\dotsc,X_n)$ and further let $\F_0:=\F_\emptyset=\{\emptyset,\Omega\}$. Thus, $\mathbb{F}:=(\F_n)_{n\in\N_0}$ is the canonical filtration of $\A$ generated by $\X$ and $\F=\sigma(\bigcup_{n\in\N_0}\F_n)$. We also denote by $\P_\F$ the restriction of $\P$ on $\F$. For $k\in\N$ we further let $\G_k:=\F_{\N\setminus\{k\}}=\sigma(X_j,j\not=k)$. 

Moreover, for some results in Subsection \ref{OU} below, we suppose that, on the same probability space $(\Om,\A,\P)$, a sequence $(Z_k)_{k\in\N}$ of independent random variables that is also independent of $(\X,\X')$ is defined in such a way that each $Z_k$, $k\in\N$, has the exponential distribution with mean one. Then, for any $t\in[0,\infty)$ we define the random sequence $\X(t)=(X_k(t))_{k\in\N}$ via 
\begin{equation*}
X_k(t):=\begin{cases}
X_k, & Z_k>t\\
X_k',& Z_k\leq t
\end{cases}
\end{equation*}
for each $k\in\N$. Note that $\X(t)$ has the same distribution as $\X$ for any $t\geq0$ and that $\X(0)=\X$. Moreover, it is not difficult to see that $(\X(t))_{t\geq0}$ is in fact a Markov process with state space $(E,\B)$.

As usual, by $L^r(\mu)$, $r\in[0,\infty)$, we denote the space of all measurabe functions $f:(E,\B)\rightarrow(\R,\B(\R))$ s.t. $\int_E |f|^rd\mu<\infty$. Furthermore, $L^\infty(\mu)$ denotes the class of all $\mu$-essentially bounded, $\B-\B(\R)$-measurable functions. Here and in what follows, $\B(\R)$ denotes the Borel-$\sigma$-field on $\R$.

As is customary, we do not distinguish between functions and equivalence classes of a.e. identical functions here. Moreover, for $r\in[0,\infty)$ we let 
$L^r_\X:=L^r(\Om,\F,\P)$ be the space of all $\sigma(\X)$-measurable random variables $F:\Om\rightarrow\R$ such that $\E|F|^r<\infty$ and, as usual, for $r\in[1,\infty)$ we let $\norm{F}_r:=(\E|F|^r)^{1/r}$. We further denote $L^\infty_\X:=L^\infty(\Om,\F,\P)$ the space of all $\P$-essentially bounded $F\in L^0(\Om,\F,\P)$ and, for $F\in L^\infty_\X$ we let $\norm{F}_\infty$ denote its essential supremum norm. By the factorization lemma, for any $F\in L^r_\X$, $r\in[0,\infty]$, there exists an $f\in L^r(\mu)$ such that $F=f\circ\X$. Such a function $f$ will be called a \textit{representative} of $F$ in what follows. Note that such an $f$ is $\mu$-a.s. unique.

 A random variable $F\in  L^2_\X$ is called \textit{cylindrical}, if $F$ is $\F_n$-measurable for some $n\in\N$, i.e. if $F$ only depends on finitely many coordinates of $\X$. Following \cite{DH} we denote by $\mathcal{S}$ the linear subspace of $ L^2_\X$ containing all cylindrical random variables.  

We further denote by $\kappa$ the counting measure on $(\N,\Pot(\N))$, i.e. $\kappa(B)=|B|$ for all $B\subseteq\N$ and below we will also make use of the product measure $\kappa\otimes\P_\F$ on the space 
$(\N\times\Omega,\Pot(\N)\otimes\F)$. Measurable functions defined on this space will be called \textit{processes} in what follows. Here, $\Pot(\N)$ denotes the power set of $\N$ and we further write $\Pot_{fin}(\N)$ for the (countable) collection of all finite subsets of $\N$. Note that a process $U$ may be identified with a sequence $(U_k)_{k\in\N}$ of $\F-\B(\R)$- measurable random variables $U_k$, $k\in\N$. The expectation operator with respect to $\P$ or $\P_\F$ will always be denoted by $\E$. For $p\in\N$ we denote by $\N^p_{\not=}$ (respectively, by $[n]^p_{\not=}$) the set of all tuples $(i_1,\dotsc,i_p)\in\N^p$ (all tuples $(i_1,\dotsc,i_p)\in[n]^p$) such that $i_j\not=i_l$ for all $j\not=l$.

\subsection{Infinite Hoeffding decompositions}\label{Hoeffding}
In this subsection we state a fundamental result about infinite Hoeffding decompositions for random variables in $L^2_\X$. 
To this end, we first review the Hoeffding decomposition for functions of finitely many independent random variables and introduce some additional useful notation that will be used throughout this work.
We refer the reader to the monographs \cite{serfling, major, KB-book, Lee} and to the papers \cite{KR, vitale, vanZwet} for basic facts about Hoeffding decompositions for functions of finitely many independent random variables.

To begin with, let us merely assume that $F\in L^1_\X$.
Then, we can define the corresponding \textit{L\'{e}vy martingale} $(F_n)_{n\in\N}$ with respect to the filtration $\mathbb{F}:=(\F_n)_{n\in\N}$ by 
\[F_n:=\E\bigl[F\,|\,\F_n\bigl]\,, \quad n\in\N\,.\]
From martingale theory it is well-known that, as $n\to\infty$, $F_n$ converges to $F$ $\Prob$-a.s. and in $L^1(\Prob)$. Furthermore, if in fact $F\in L^r_\X$ for some $r\in(1,\infty)$, then 
$\sup_{n\in\N}\E\abs{F_n}^r<\infty$ and the convergence also takes place in $L^r(\Prob)$. As $F_n$ is $\F_n$-measurable, by the factorization lemma, there exist measurable functions 
\begin{equation*}
 g_n: \Bigl(\prod_{j=1}^n E_j,\bigotimes_{j=1}^n \B_j\Bigr)\rightarrow\bigl(\R,\B(\R)\bigr)\,,\quad n\in\N\,,
\end{equation*}
such that $F_n=g_n(X_1,\dotsc,X_n)$ for each $n\in\N$. Hence, we have the following \textit{Hoeffding decomposition} of $F_n$:
\begin{equation}\label{hdyn}
 F_n=\sum_{M\subseteq[n]} F_{n,M}\,,
\end{equation}
where we constantly write $[n]:=\{1,\dotsc,n\}$ and where the \textit{Hoeffding components} $F_{n,M}$, $M\subseteq[n]$, are $\Prob$-a.s. uniquely determined by (a) and (b) as follows:
\begin{enumerate}[(a)]
 \item For each $M\subseteq[n]$, the random variable $F_{n,M}$ is measurable with respect to $\F_M=\sigma(X_j,j\in M)$.
 \item For all $M,K\subseteq[n]$ we have $\E[F_{n,M}\,|\,\F_K]=0$ $\Prob$-a.s. unless $M\subseteq K$.
\end{enumerate}
Note that, due to independence, for $M,K\subseteq[n]$ we always have that 
\begin{equation*}
 \E[F_{n,M}\,|\,\F_K]=\E[F_{n,M}\,|\,\F_{K\cap M}]
\end{equation*}
so that we could replace (b) with
\begin{enumerate}[(c)]
 \item For all $M,L\subseteq[n]$ such that $L\subsetneq M$ we have $\E[F_{n,M}\,|\,\F_L]=0$ $\Prob$-a.s.
\end{enumerate}

The following inclusion-exclusion type formula for the Hoeffding components is well known:
\begin{align}\label{hdform}
 F_{n,M}&=\sum_{L\subseteq M}(-1)^{\abs{M}-\abs{L}}\E\bigl[F_n\,\bigl|\,\F_L\bigr]\notag\\
 &=\sum_{L\subseteq M}(-1)^{\abs{M}-\abs{L}}\E\bigl[F\,\bigl|\,\F_L\bigr]\,,\quad M\subseteq[n]\,,
\end{align}
where the second identity follows from 
\begin{equation*}
 \E\bigl[F_n\,\bigl|\,\F_L\bigr]=\E\Bigl[\E\bigl[F\,\bigl|\,\F_n\bigr]\,\Bigl|\,\F_L\Bigr]=\E\bigl[F\,\bigl|\,\F_L\bigr]\,,\quad L\subseteq[n]\,.
\end{equation*}
From now on, we will assume that, in fact, $F\in L^2_\X$. Then, the same holds for $F_n$ for each $n\in\N$ and, by \eqref{hdform}, also for its Hoeffding components $F_{n,M}$, $M\subseteq[n]$, which are in fact known to be pairwise orthogonal or uncorrelated in $L^2(\P)$ in this case. Another important consequence of \eqref{hdform} is that, for any finite subset $M$ of $\N$, the Hoeffding components $F_{n,M}$ are stationary for $n\geq \max(M)$, i.e. we have 
\begin{equation}\label{cons1}
F_{n,M}=F_{l,M}
\end{equation}
whenever $1\leq n<l$ and $M\subseteq[n]$. In particular, for each finite subset $M\subseteq\N$ and all $n\geq\max(M)$ we have 
\begin{equation}\label{cons3}
 F_{n,M}=F_{\max(M),M}\,.
\end{equation}
Henceforth, we may and will thus only write $F_M$ for $F_{\max(M),M}$.

For $p\in\N_0 ,n\in\N$ with $n\geq p$ define 
\begin{equation*}
F_n^{(p)}:=\sum_{\substack{M\subseteq[n]:\\ \abs{M}=p}} F_{M}\,.
\end{equation*}
Then, from \eqref{cons1} we conclude that 
\begin{equation}\label{cons2}
F_{n+1}^{(p)}=F_n^{(p)}+\sum_{\substack{M\subseteq[n+1]:\\ \abs{M}=p,\, n+1\in M}} F_{M}
\end{equation}
for all $p,n\in\N$ with $n\geq p$. Now it follows from (b) above that 
\[\E\bigl[F_{M}\,\bigl|\,\F_n\bigr]=\E\bigl[F_{n+1,M}\,\bigl|\,\F_n\bigr]=0\]
whenever $M\subseteq[n+1],\abs{M}=p$ and $n+1\in M$. Hence, from \eqref{cons2} we infer that, for fixed $p\in\N$, the sequence 
$(F_n^{(p)})_{n\geq p}$ is also a martingale with respect to $(\F_n)_{n\geq p}$. 
If $F\in L^2_\X$, then the martingales $(F_n)_{n\in\N}$ and $(F_n^{(p)})_{n\in\N}$ are in $L^2_\X$ as well and we have the bounds 
\begin{align}
\sup_{n\in\N}\E\bigl[F_n^2\bigr]&=\sup_{n\in\N}\E\Bigl[\Bigl(\E\bigl[F\,\bigl|\,\F_n\bigr]\Bigr)^2\Bigr]\leq \E\bigl[F^2\bigr]\quad\text{and}\label{l2bound1}\\
\sup_{p\in\N_0}\sup_{\substack{n\in\N:\\ n\geq p}}\Var\bigl(F_n^{(p)}\bigr)&\leq\sup_{n\in\N}\Var\bigl(F_n\bigr)=\sup_{n\in\N}\E\bigl[F_n^2\bigr]-\bigl(\E[F]\bigr)^2\leq \Var(F)\label{l2bound2}\,,
\end{align}
where we have used the orthogonality of the Hoeffding decomposition to obtain
\begin{equation*}
 \Var\bigl(F_n^{(p)}\bigr)=\sum_{\substack{M\subseteq[n]:\\ \abs{M}=p}}\Var\bigl(F_{M}\bigr)\leq \sum_{M\subseteq[n]}\Var\bigl(F_{M}\bigr)=\Var(F_n)\,.
\end{equation*}
From \eqref{l2bound2} and martingale theory we conclude that, for fixed $p\in\N_0$, $(F_n^{(p)})_{n\geq p}$ converges to some random variable $F^{(p)}\in L^2_\X$ both $\P$-a.s. and in $L^2(\P)$. 
Moreover, due to the convergence in $L^2(\P)$, for $p\not=q$ we have 
\begin{equation}\label{orthrel}
 \E\bigl[F^{(p)} F^{(q)}\bigr]=\lim_{n\to\infty}\E\bigl[F_n^{(p)} F_n^{(q)}\bigr]=0\,,
\end{equation}
where we have again used the orthogonality of the Hoeffding components for the second equality. This leads to the natural question if the identity 
\begin{equation}\label{Fsum}
 F=\sum_{p=0}^\infty F^{(p)}
\end{equation}
holds in the $L^2(\P)$-sense. In this case \eqref{Fsum} can be considered an \textit{infinite orthogonal Hoeffding decomposition} of $F$ in $L^2(\P)$. 
In fact, one has the following result, whose statement and an outline of a proof have been given in the introduction of \cite{kwapien}. A similar approach to this decomposition can be found in the proof of \cite[Lemma 2.4]{Duer}. Starting from the preparations above, it may be proved from standard facts about unconditional convergence of series with orthogonal summands in Hilbert spaces. Since this result is of crucial importance for the theory in this work, we provide a complete and purely probabilistic proof in Section \ref{appendix}.   

\begin{prop}[Infinite Hoeffding decomposition]\label{infhoeff}
Suppose that $F\in L^2_\X$. Then, using the above notation, the following orthogonal and unconditional series expansions hold in the $L^2(\P)$-sense:
\begin{align}
F^{(p)}&=\sum_{\substack{M\subseteq\N:\\ \abs{M}=p}}F_{M},\quad 
p\in\N_0\label{genhd1}\,,\\
F&=\sum_{M\in\Pot_{fin}(\N)}F_{M} 
=\sum_{p=0}^\infty F^{(p)}\label{genhd2}\,.
\end{align}
The representations of $F$ in \eqref{genhd2} are $\P$-a.s. unique and both are called the \textit{infinite Hoeffding decomposition} of $F$ in $L^2_\X$. 
\end{prop}

For $p\in\N_0$ the space $\mathcal{H}_p=\mathcal{H}_{p,\X}$ consisting of all random variables $F\in L^2_\X$ whose infinite Hoeffding decomposition is of the form
\[F=F^{(p)}=\sum_{M\subseteq\N: |M|=p}F_M\]
is called the \textit{$p$-th Hoeffding space} associated with $\X$. It is easy to see that each $\mathcal{H}_p$, $p\in\N_0$, is a closed linear subspace of $L^2_\X$.
Let us denote by $J_p$ the orthogonal projection on $\mathcal{H}_p$, $p\in\N_0$. Thus, if $F\in L^2_\X$ has the infinite Hoeffding decomposition \eqref{genhd2}, then 
\begin{equation}\label{intformHD}
 F=\sum_{p=0}^\infty J_p(F)=\E[F]+\sum_{p=1}^\infty J_p(F),
\end{equation}
where $J_p(F)=F^{(p)}$, $p\in\N_0$.
We will see in the next subsection that the infinite Hoeffding decomposition plays the same role in our setting as the Wiener-It\^{o} chaos decomposition does for the Malliavin calculus on Gaussian \cite{Nualart} and Poisson spaces \cite{Lastsa} or for Rademacher sequences \cite{Privault}. Therefore, for $p\in\N_0$, we also call an $F\in\mathcal{H}_p$ a \textit{$p$-th chaos}. 
If $p\geq1$ and for a finite vector $\X=(X_1,\dotsc,X_n)$ of independent random variables, in statistical theory such a random variable is also called a (not necessarily symmetric) \textit{completely degenerate $U$-statistic of order $p$}. We continue to use this notion in our setting of an infinite underlying sequence $\X$.

\subsection{Malliavin operators and chaotic decomposition}\label{mallop}

In this subsection we review the definitions and domains of the three Malliavin operators $D,\delta$ and $L$ given in \cite{DH}. Moreover, we extend the theory therein by showing that the infinite Hoeffding decomposition from Proposition \ref{infhoeff} plays the role of a chaotic decomposition for random variables $F\in L^2_\X$. In particular, we derive new and explicit characterizations of the domains for $D,\delta$ and $L$ in terms and by means of this decomposition. On the one hand these criteria facilitate the verification, whether the random variable under consideration belongs to the respective domain. On the other hand, several novel important structural properties of the operators $D$ and $L$ are proved via the infinite Hoeffding decomposition. We begin by reviewing the relevant definitions and theory from \cite{DH} before stating and proving our various extensions.\smallskip\\

For $k\in\N$ and $F\in L^1_\X$ with representative $f\in L^1(\mu)$ we define 
\begin{equation*}
D_kF:=F-\E[F\,|\,\G_k]=F-\E\bigl[f\bigl(\X^{(k)}\bigr)\,\bigl|\,\X\bigr]=\E\bigl[f(\X)-f(\X^{(k)})\,\bigl|\,\X\bigr], 
\end{equation*}
where we recall that $\G_k=\F_{\N\setminus\{k\}}=\sigma(X_j,j\not=k)$.
It is easy to see that each $D_k$ defines a self-adjoint operator with norm one on $L^2_\X$. Indeed, $D_k$ is the orthogonal projection on the space of all random variables in $L^2_\X$ that really depend on the random variable $X_k$. It thus follows that $D_kD_k=D_k$ and $D_kD_l=D_lD_k$ for all $k,l\in\N$. Furthermore, $F\in L^r_\X$ implies $D_kF\in L^r_\X$ and $\norm{D_kF}_r\leq 2\norm{F}_r$ for all $r\in[1,\infty)$. Also note that the symmetry relation $\E[FD_kG]=\E[GD_kF]$ continues to hold for $G\in L^r_\X$ and $F\in L^s_\X$, if $r,s\in[1,\infty]$ are such that $r^{-1}+s^{-1}=1$ (with the usual convention that $1/\infty:=0$).

Moreover, we let $\tilde{D}F:=(D_kF)_{k\in\N}$ denote the corresponding process. This definition makes sense for any $F\in L^2_\X$ (in fact even for any $F\in L^1_\X$) although $\tilde{D}F$ need not necessarily be an element of $L^2(\kappa\otimes\P_\F)$. Thus, in \cite{DH} the \textit{Malliavin derivative} $D$ is first defined on the $L^2_\X$-dense subspace $\mathcal{S}$ via $DF=\tilde{D}F$, which is always in $L^2(\kappa\otimes\P_\F)$ as $D_kF=0$ if $F$ is $\F_n$-measurable and $k>n$ (see \cite[Lemma 3.1]{DH}). Then, it is proved in \cite[Corollary 2.5]{DH} that this operator on $\mathcal{S}$ is closable and its closure is again denoted by $D$. We write $\dom(D)$ for the domain of this closure, which therefore defines a densely defined, closed operator $D:\dom(D)\subseteq L^2_\X\rightarrow L^2(\kappa\otimes\P_\F)$. In \cite[Lemma 2.6]{DH}, a sufficient condition for $F\in L^2_\X$ to be in $\dom(D)$ is given but no explicit characterization of $\dom(D)$ is provided. Moreover, it is not clear from the outset that, for $F\in\dom(D)\setminus\mathcal{S}$, this operator is still given by $DF=(D_kF)_{k\in\N}$, although we will see that this indeed holds in Lemma \ref{domDle} below. Since $\dom(D)$ is in general a strict subset of $L^2_\X$, it is thus desirable to have verifiable equivalent conditions for $F$ to be in $\dom(D)$. These are provided by Proposition \ref{domD} below. The following slight extension of \cite[Lemma 3.2]{DH} will be used in what follows without further mention: For $F\in L^2_\X$, $k\in\N$ and $M\subseteq \N$ one has
\begin{align*}
 D_k\E\bigl[F\,\big|\,\F_M\bigr]=\E\bigl[F\,\big|\,\F_M\bigr]-\E\bigl[F\,\big|\,\F_M\cap \G_k\bigr]=\E\bigl[D_kF\,\big|\,\F_M\bigr].
\end{align*}

\medskip

As usual, the \textit{divergence operator} $\delta$ has been defined in \cite{DH} as the adjoint operator of $D$. In particular, by definition, its domain $\dom(\delta)$ is the collection of all processes $U=(U_k)_{k\in\N}\in L^2(\kappa\otimes\P_\F)$ such that there is a constant $C\in[0,\infty)$ (depending on $U$) with the property that 
\[\Babs{\langle DF,U\rangle_{L^2(\kappa\otimes\P_\F)}}=\Bbabs{\sum_{k=1}^\infty \E\bigl[D_kF U_k\bigr]}\leq C\norm{F}_2\]
holds for any $F\in\dom(D)$. For $U=(U_k)_{k\in\N}\in\dom(\delta)$ and $F\in\dom(D)$ one thus has the \textit{Malliavin integration by parts formula}
\begin{equation}\label{intpartsM}
\langle DF,U\rangle_{L^2(\kappa\otimes\P_\F)}=\E\bigl[F\delta U\bigr]
\end{equation} 
which, as $\langle DF,U\rangle_{L^2(\kappa\otimes\P_\F)}=\sum_{k=1}^\infty \E[D_kF U_k]=\sum_{k=1}^\infty \E[F D_kU_k]$, entails that 
\begin{equation}\label{formdelta}
\delta U=\sum_{k=1}^\infty D_kU_k.
\end{equation}
Note that for the latter argument to be sound, one actually has to make sure that interchanging the infinite sum and the expectation is justified. We refer to Proposition \ref{domdelta} below for a rigorous proof of the formula \eqref{formdelta} as well as for new explicit characterizations of $\dom(\delta)$.
\smallskip\\

The third Malliavin operator defined in \cite{DH} is the \textit{Ornstein-Uhlenbeck generator} or \textit{number operator} $L$ corresponding to $\X$. By definition, its domain $\dom(L)$ is the collection of all $F\in\dom(D)$ such that $DF\in\dom(\delta)$ and then, for $F\in\dom(L)$, one defines 
\[LF=-\delta DF=-\sum_{k=1}^\infty D_kF=\sum_{k=1}^\infty \E\bigl[f(\X^{(k)})-f(\X)  \,\bigl|\,\X\bigr],\]
where the explicit formula \eqref{formdelta} for $\delta U$ as well as $D_kD_k=D_k$ have been used for the second equality. 
\smallskip\\

The remainder of this section is devoted to our extensions of the theory developed in \cite{DH} and in \cite{Duer}. Most of the results to follow are of independent interest but will be particularly useful for proving error bounds on distributional approximations in follow up works.

\begin{lemma}\label{domDle}
Suppose that $F_n,F\in L^2_\X$, $n\in\N$, are such that $(F_n)_{n\in\N}$ converges in $L^2_\X$ to $F$. Then, for any $k\in\N$, the sequence $(D_kF_n)_{n\in\N}$ converges to $D_kF$ in $L^2_\X$. Moreover, if $F\in\dom(D)$, then $DF=\tilde{D}F=(D_kF)_{k\in\N}$.
\end{lemma}

\begin{proof}
The first claim follows from $D_kF_n=F_n-\E[F_n|\G_k]$ and $D_kF=F-\E[F|\G_k]$ via well-known properties of conditional expectations. For the second claim let $DF=(V_k)_{k\in\N}\in L^2(\kappa\otimes\P_\F)$. If $F\in\dom(D)$, then there is a sequence $F_n\in\mathcal{S}$, $n\in\N$, such that $F_n\rightarrow F$ in $L^2(\P)$ and $DF_n\rightarrow DF$ in $L^2(\kappa\otimes\P_\F)$ as $n\to\infty$. Thus, for each fixed $k\in\N$, 
\[0\leq \E\Bigl[\bigl(D_kF_n-V_k\bigr)^2\Bigr]\leq\sum_{l=1}^\infty\E\Bigl[\bigl(D_lF_n-V_l\bigr)^2\Bigr]=\int_{\N\times\Omega} (DF_n-DF)^2d(\kappa\otimes\P_\F)\stackrel{n\to\infty}{\longrightarrow}0.\]
Since $(D_kF_n)_{n\in\N}$ converges also to $D_kF$ in $L^2_\X$ by the first claim, it follows that $V_k=D_kF$ $\P$-a.s. for each $k\in\N$. Since $\N$ is countable, this implies that 
$(V_k)_{k\in\N}=(D_kF)_{k\in\N}$ $\P$-a.s.
\end{proof}

\begin{prop}\label{HoeffD}
Let $F\in L^2_\X$ have the infinite Hoeffding decomposition \eqref{genhd2}. Then, for any $k\in\N$, the infinite Hoeffding decomposition of $D_kF$ is given by\\ 
$D_kF=\sum_{M\in\Pot_{fin}(\N): k\in M} F_M$.
\end{prop}

\begin{proof}
Note that, by the defining properties of the Hoeffding components $F_M$, we have $\E[F_M|\G_k]=\1_{\{k\notin M\}} F_M$ $\P$-a.s. for any $k\in\N$ and any $M\in\Pot_{fin}(\N)$. Hence, using that limits in $L^2(\P)$ and conditional expectation may be interchanged, we obtain
\begin{align*}
D_kF&=F-\E[F_M|\G_k]=\sum_{M\in\Pot_{fin}(\N)}F_M-\sum_{M\in\Pot_{fin}(\N)}\E\bigl[F_M\,\bigl|\,\G_k\bigr]\\
&=\sum_{M\in\Pot_{fin}(\N)}F_M-\sum_{M\in\Pot_{fin}(\N)}\1_{\{k\notin M\}}F_M
=\sum_{M\in\Pot_{fin}(\N): k\in M} F_M.
\end{align*}
\end{proof}

The following new formula is analogous to the celebrated Stroock formula on Gaussian spaces and similar versions in Poisson and Rademacher seetings. 

\begin{cor}[Stroock type formula]\label{stroock}
 Let $F\in L^2_\X$ have the infinite Hoeffding decomposition \eqref{genhd2}. Then, for any 
 $M=\{i_1,\dotsc,i_p\}\in\Pot_{fin}(\N)$ it holds that 
 $F_M=\E\bigl[D_{i_1}D_{i_2}\ldots D_{i_p}F\,\bigl|\, \F_M\bigr]$. Thus, one has the explicit formulae
\begin{align*}
F&=\E[F]+\sum_{p=1}^\infty\sum_{1\leq i_1<i_2<\ldots<i_p<\infty}\E\bigl[D_{i_1}D_{i_2}\ldots D_{i_p}F\,\bigl|\, \F_{\{i_1,\dotsc,i_p\}}\bigr]\\
&=\E[F]+\sum_{p=1}^\infty\frac{1}{p!}\sum_{(i_1,\dotsc,i_p)\in\N^p_{\not=}}\E\bigl[D_{i_1}D_{i_2}\ldots D_{i_p}F\,\bigl|\, \F_{\{i_1,\dotsc,i_p\}}\bigr].
\end{align*}
 \end{cor}

 \begin{proof}
  By iteration, from Proposition \ref{HoeffD} we obtain that 
  \begin{align*}
  D_{i_1}D_{i_2}\ldots D_{i_p}F&= \sum_{N\in\Pot_{fin}(\N): M\subseteq N} F_N.
  \end{align*}
  Since the series here converges (unconditionally) in $L^2(\P)$, we can interchange summation and the conditional expectation with respect to $\F_M$, and then property (b) of the Hoeffding components yields the first claim. The formulae for $F$ now follow from the first claim and \eqref{genhd2} by observing that $D_kD_l=D_lD_k$ for all $k,l\in\N$. 
 \end{proof}

The next result gives novel, verifiable criteria for $F\in L^2_\X$ to be in $\dom(D)$. In particular, (ii) and (iv) will be useful in practical applications.

\begin{prop}\label{domD}
For a random variable $F\in L^2_\X$ the following conditions are equivalent:
\begin{enumerate}[{\normalfont (i)}]
\item $F\in \dom(D)$.
\item $(D_kF)_{k\in\N}\in L^2(\kappa\otimes\P_\F)$, i.e. $\sum_{k=1}^\infty \E[(D_kF)^2]<\infty$.
\item For one (and then every) representative $f\in L^2(\mu)$ of $F$ one has \\$\displaystyle\sum_{k=1}^\infty\E\Bigl[\Bigl(f\bigl(\X^{(k)}\bigr)-f\bigl(\X\bigr)\Bigr)^2 \Bigr]<\infty$.
\item With the infinite Hoeffding decomposition \eqref{genhd2} of $F$
one has that \\
$\displaystyle\sum_{p=1}^\infty p\,\E\bigl[(F^{(p)})^2\bigr]=\sum_{p=1}^\infty p\,\E\bigl[J_p(F)^2\bigr]=  \sum_{M\in \Pot_{fin}(\N)}|M|\,\E[F_{M}^2]<\infty$.
\end{enumerate}
\end{prop}

\begin{proof}
(i)$\Rightarrow$(ii): Suppose that (i) holds but (ii) does not. Then, by (i) there is a sequence $F_n\in\mathcal{S}$, $n\in\N$, such that $F_n\to F$ in $L^2(\P)$ and $DF_n=\tilde{D}F_n\to DF$ in $L^2(\kappa\otimes\P_\F)$ as $n\to\infty$. In particular, there is a $C\in(0,\infty)$ such that $\norm{DF_n}_{L^2(\kappa\otimes\P_\F)}^2= \sum_{k=1}^\infty \E[(D_kF_n)^2]\leq C$ for each $n\in\N$. On the other hand, since (ii) does not hold, there is an $m\in\N$ such that $\sum_{k=1}^m \E[(D_kF)^2]>2C$. By Lemma \ref{domDle}, from $F_n\to F$ in $L^2(\P)$ it follows that $D_kF_n\to D_kF$ in $L^2(\P)$ as $n\to\infty$ for each $k\in\N$. Therefore, 
\[\lim_{n\to\infty}\sum_{k=1}^m \E[(D_kF_n)^2]=\sum_{k=1}^m \E[(D_kF)^2]>2C\]
contradicting 
\[\sum_{k=1}^m \E[(D_kF_n)^2]\leq \sum_{k=1}^\infty \E[(D_kF_n)^2]\leq C\]
for each $n\in\N$. Hence, (i) implies (ii).\smallskip\\
(ii)$\Leftrightarrow$(iv): From Proposition \ref{HoeffD} we know that $D_kF=\sum_{M\in\Pot_{fin}(\N):k\in M}F_M$ for each $k\in\N$. Thus, from the orthogonality of the Hoeffding components we have
\begin{align*}
\sum_{k=1}^\infty \E[(D_kF)^2]&=\sum_{k=1}^\infty \Bigl(\sum_{M\in\Pot_{fin}(\N):k\in M}\E\bigl[F_M^2\bigr]\Bigr)=\sum_{M\in\Pot_{fin}(\N)}\E\bigl[F_M^2\bigr]\sum_{k\in M}1\\
&=\sum_{M\in\Pot_{fin}(\N)}|M|\,\E\bigl[F_M^2\bigr]=\sum_{p=1}^\infty p\,\E\bigl[(F^{(p)})^2\bigr],
\end{align*}
where the last equality follows from \eqref{genhd1} and the (general) rearrangement theorem for real series with nonnegative summands.\smallskip\\
(ii)$\Leftrightarrow$(iii): On the one hand, for $k\in\N$, we have 
\begin{align*}
\E\bigl[(D_kF)^2\bigr]&=\E\Bigl[F^2+\bigl(\E\bigl[F\,\big|\,\G_k\bigr]\bigr)^2-2F\bigl(\E\bigl[F\,\big|\,\G_k\bigr]\bigr)\Bigr]
=\E[F^2]-\E\Bigl[\bigl(\E\bigl[F\,\big|\,\G_k\bigr]\bigr)^2\Bigr]
\end{align*}
and, on the other hand, recalling that $\X$ and $\X^{(k)}$ have the same distribution we have 
\begin{align}\label{eqdomd}
&\E\Bigl[\Bigl(f\bigl(\X^{(k)}\bigr)-f\bigl(\X\bigr)\Bigr)^2 \Bigr]=\E\Bigl[f\bigl(\X^{(k)}\bigr)^2+f\bigl(\X\bigr)^2-2f\bigl(\X^{(k)}\bigr)f\bigl(\X\bigr)\Bigr]\notag\\
&=2\E[F^2]-2\E\Bigl[f\bigl(\X\bigr)\E\bigl[f\bigl(\X^{(k)}\bigr)\,\big|\,\X\bigr]\Bigr]=2\E[F^2]-2\E\Bigl[F\,\E\bigl[F\,\big|\,\G_k\bigr]\Bigr]\notag\\
&=2\E[F^2]-2\E\Bigl[\bigl(\E\bigl[F\,\big|\,\G_k\bigr]\bigr)^2\Bigr]=2\E\bigl[(D_kF)^2\bigr],
\end{align}
proving the claim.\smallskip\\
(iv)$\Rightarrow$(i): For $n\in\N$ let $F_n:=\E[F|\F_n]\in\mathcal{S}$. Then, we know that $F_n\to F$ in $L^2(\P)$ as $n\to\infty$ and by the definition of the closure of an operator it suffices to show that $(DF_n)_{n\in\N}$ converges in $L^2(\kappa\otimes\P_\F)$, as then $F\in \dom(D)$ and $DF=\lim_{n\to\infty} DF_n$ in $L^2(\kappa\otimes\P_\F)$. But, for integers $n>m\geq1$ we have 
\begin{align*}
&\sum_{k=1}^\infty\E\Bigl[\bigl(D_kF_n-D_kF_m\bigr)^2\Bigr]=\sum_{k=1}^\infty\sum_{\substack{M\subseteq[n]:\\k\in M, M\cap[m+1,n]\not=\emptyset}}\E\bigl[F_M^2\bigr]
=\sum_{\substack{M\subseteq[n]:\\k\in M, M\cap[m+1,n]\not=\emptyset}}|M|\E\bigl[F_M^2\bigr]\notag\\
&=\sum_{\substack{M\subseteq[n]:\\k\in M}}|M|\E\bigl[F_M^2\bigr]-\sum_{\substack{M\subseteq[m]:\\k\in M}}|M|\E\bigl[F_M^2\bigr]\longrightarrow0,
\end{align*}
as $n,m\to\infty$ by (iii). Hence, $(DF_n)_{n\in\N}$ is a Cauchy sequence in the Hilbert space $L^2(\kappa\otimes\P_\F)$ and therefore convergent.
\end{proof}

The following proposition similarly characterizes the domain of the adjoint $\delta$ of $D$.

\begin{prop}\label{domdelta}
For a process $U=(U_k)_{k\in\N}\in L^2(\kappa\otimes\P_\F)$ the following statements are equivalent: 
\begin{enumerate}[{\normalfont (i)}]
  \item $U\in\dom(\delta)$.
	\item If $\sum_{M\in\Pot_{fin}(\N)}U_{k,M}$ is the infinite Hoeffding decomposition of $U_k\in L^2_\X$, $k\in\N$, then $\displaystyle\sum_{M\in\Pot_{fin}(\N)}\E\Bigl[\Bigl(\sum_{k\in M}U_{k,M}\Bigr)^2\Bigr]<\infty$.
  \item The series $\sum_{k=1}^\infty D_kU_k$ converges in $L^2_\X$.
	\item $\displaystyle\sup_{m\in\N}\E\biggl[\Bigl(\sum_{k=1}^m D_kU_k\Bigr)^2\biggr]<\infty$.
  \end{enumerate}
	In this case, one has $\delta U=\sum_{k=1}^\infty D_kU_k$ in $L^2(\P)$.
\end{prop}

\begin{remark}\label{deltarem}
It is remarkable that, in the present product space setting, there is a \textit{pathwise} formula $\delta U=\sum_{k=1}^\infty D_kU_k$ holding for any process $U=(U_k)_{k\in\N}\in\dom(\delta)$. This is contrary to the Poisson setting, where such a formula has only been proved under additional integrability assumptions (see e.g. \cite[Theorem 6]{Lastsa}). Moreover, an analogous characterization of $\dom(\delta)$ to (iii) or (iv) is not known on Poisson spaces.  
\end{remark}

\begin{proof}[Proof of Proposition \ref{domdelta}]
(i)$\Rightarrow$(ii): Suppose that (ii) does not hold. Then, for any $n\in\N$, there are finitely many pairwise distinct sets $M_1,\dotsc,M_{m_n}\in\Pot_{fin}(\N)$ such that 
\[\sum_{j=1}^{m_n}\E\Bigl[\Bigl(\sum_{k\in M_j}U_{k,M_j}\Bigr)^2\Bigr]>n,\quad n\in\N.\]
For $n\in\N$, define $F_n:=\sum_{j=1}^{m_n}\sum_{l\in M_j} U_{l,M_j}$ which is certainly in $\dom(D)$ by Proposition \ref{domD}. Moreover, by the definition of $\delta$, orthogonality, the symmetry of the $D_k$, $k\in\N$, and by Proposition \ref{HoeffD}, for any $n\in\N$,
\begin{align*}
\Babs{\E\bigl[F_n \delta U\bigr]}&=\Babs{\sum_{k=1}^\infty \E\bigl[D_kF_n U_k\bigr]}=\Babs{\sum_{k=1}^\infty \E\bigl[F_n D_k U_k\bigr]}\\
&=\Babs{\sum_{k=1}^\infty \sum_{j=1}^{m_n}\sum_{l\in M_j}\1_{\{k\in M_j\}}\E\bigl[U_{k,M_j}U_{l,M_j}\bigr]}=\sum_{j=1}^{m_n}\E\Bigl[\Bigl(\sum_{k\in M_j}U_{k,M_j}\Bigr)^2\Bigr]\\
&=\norm{F_n}_2^2>\sqrt{n}\norm{F_n}.
\end{align*}
Hence, $U$ cannot be in $\dom(\delta)$ and thus (i) does not hold, either.\smallskip\\
(ii)$\Rightarrow$(iii): Since $L^2_\X$ is a Hilbert space, it suffices to show that the series $\sum_{k=1}^\infty D_kU_k$ is Cauchy in $L^2_\X$. Using Proposition \ref{HoeffD} and orthogonality, for integers $n>m\geq1$ we have 
\begin{align}\label{domdel1}
\E\biggl[\Bigl(\sum_{k=m+1}^n D_kU_k\Bigr)^2\biggr]&=\sum_{k,l=m+1}^n \E\bigl[D_kU_k D_lU_l\bigr]=\sum_{k,l=m+1}^n\sum_{M\in\Pot_{fin}(\N):k,l\in M}\E\bigl[U_{k,M}U_{l,M}\bigr]\notag\\
&=\sum_{M\in\Pot_{fin}(\N)}\E\Bigl[\Bigl(\sum_{k\in M\cap[m+1,n]}U_{k,M}\Bigr)^2\Bigr]\stackrel{m,n\to\infty}{\longrightarrow}0
\end{align} 
by (ii) and the dominated convergence theorem.\smallskip\\
(iii)$\Rightarrow$(i): If (iii) holds, then $C:=\norm{\sum_{k=1}^\infty D_kU_k}_2$ is a well-defined and finite constant and, thanks to the symmetry of $D_k$, $k\in\N$, for any $F\in\dom(D)$ one further has that 
\begin{align*}
\Babs{\langle DF,U\rangle_{L^2(\kappa\otimes\P_\F)}}&=\Bbabs{\sum_{k=1}^\infty \E\bigl[D_kF U_k\bigr]}=\Bbabs{\sum_{k=1}^\infty \E\bigl[F D_kU_k\bigr]}
=\Bbabs{\E\biggl[F \sum_{k=1}^\infty D_kU_k\biggr]}\\
&\leq \norm{F}_2 \BEnorm{\sum_{k=1}^\infty D_kU_k}=C\norm{F}_2,
\end{align*}
implying $U\in\dom(\delta)$. Note that the $L^2(\P)$-convergence of the series has been used to obtain the third equality.\smallskip\\
(iii)$\Rightarrow$(iv): This is clear.\smallskip\\
(iv)$\Rightarrow$(ii): Let $S<\infty$ be the supremum in (iv). As in \eqref{domdel1}, for $m\in\N$, we have  
\begin{align}\label{domdel2}
+\infty&>S\geq\E\biggl[\Bigl(\sum_{k=1}^m D_kU_k\Bigr)^2\biggr]=\sum_{M\in\Pot_{fin}(\N)}\E\Bigl[\Bigl(\sum_{k\in M\cap[m]}U_{k,M}\Bigr)^2\Bigr].
\end{align} 
If the sum in (ii) were not finite, then there would exist finitely many $M_1,\dotsc,M_r\in\Pot_{fin}(\N)$ such that 
\[\sum_{j=1}^r\E\Bigl[\Bigl(\sum_{k\in M_j}U_{k,M_j}\Bigr)^2\Bigr]\geq S+1.\]
But then, choosing $m\geq\max(M_1\cup\ldots\cup M_r)$ in \eqref{domdel2} would lead to the contradiction $S\geq S+1$, which is impossible for finite $S$. Hence (iv) implies (ii).
\smallskip\\
To prove the claimed explicit formula for $\delta U$ recall that, by \eqref{intpartsM}, $\delta U$ is characterized by
\begin{equation*}
\E\bigl[F\delta U\bigr]=\langle DF,U\rangle_{L^2(\kappa\otimes\P_\F)}=\sum_{k=1}^\infty \E\bigl[D_kF U_k\bigr]=\lim_{n\to\infty} \E\bigl[F \sum_{k=1}^nD_kU_k\bigr],
\end{equation*}
for all $F\in\dom(D)$, where the last equality holds again due to the symmetry of $D_k$, $k\in\N$. Now, since $F\in L^2_\X$, by (iii) we can interchange the limit and the expectation here to obtain
\[\E\bigl[F\delta U\bigr]=\lim_{n\to\infty} \E\bigl[F \sum_{k=1}^nD_kU_k\bigr]=\E\bigl[F \sum_{k=1}^\infty D_kU_k\bigr],\]
implying that $\delta U=\sum_{k=1}^\infty D_kU_k$ since $\dom(D)$ is dense in $L^2_\X$.
\end{proof}

 As for the Malliavin derivative $D$ and the divergence $\delta$, we can give more explicit characterizations of $\dom(L)$.

\begin{prop}\label{domL}
 For $F\in L^2_\X$ with the infinite Hoeffding decomposition \eqref{genhd2} the following conditions are equivalent:
 \begin{enumerate}[{\normalfont (i)}]
  \item $F\in\dom(L)$.
  \item The series $\sum_{k=1}^\infty D_kF$ converges in $L^2_\X$. 
  \item $\displaystyle \sup_{m\in\N}\E\Bigl[\Bigl(\sum_{k=1}^m D_kF\Bigr)^2\Bigr]<\infty$.
	\item $\displaystyle\sum_{p=1}^\infty p^2\,\E\bigl[(F^{(p)})^2\bigr]=\sum_{p=1}^\infty p^2\,\E\bigl[J_p(F)^2\bigr] = \sum_{M\in \Pot_{fin}(\N)}|M|^2\,\E[F_{M}^2]<\infty$.
 \end{enumerate}
 In this case, $LF=-\sum_{k=1}^\infty D_kF$ has the infinite Hoeffding decomposition \\
 $\displaystyle LF=-\sum_{M\in\Pot_{fin}(\N)}|M|\, F_M =-\sum_{p=1}^\infty p\,F^{(p)}=-\sum_{p=1}^\infty p\,J_p(F)$.
\end{prop}

\begin{proof}
(i)$\Rightarrow$(ii): If $F\in\dom(L)$, then, by definition, $F\in\dom(D)$ and $DF=(D_kF)_{k\in\N}\in\dom(\delta)$. Since $D_kD_k=D_k$, by Proposition \ref{domdelta} this in particular implies that 
$\sum_{k=1}^\infty D_kF$ converges in $L^2_\X$.\smallskip\\
(ii)$\Rightarrow$(iii): This is clear. \smallskip\\
(iii)$\Rightarrow$(iv): Using orthogonality, for each fixed $m\in\N$ we have 
\begin{align*}
 & \E\Bigl[\Bigl(\sum_{k=1}^m D_kF\Bigr)^2\Bigr]
=\E\biggl[\Bigl(\sum_{k=1}^m \sum_{M\in\Pot_{fin}(\N):k\in M}F_M\Bigr)^2\biggr]\\
&=\sum_{k,l=1}^m\E\biggl[ \sum_{\substack{M,N\in\Pot_{fin}(\N):\\
 k\in M, l\in N}}F_M F_N\biggr] =\sum_{k,l=1}^m \sum_{\substack{M,N\in\Pot_{fin}(\N):\\
 k\in M, l\in N}}\E\bigl[F_M F_N\bigr]\\
&=\sum_{k,l=1}^m \sum_{\substack{M\in\Pot_{fin}(\N):\\ k,l\in M}}\E\bigl[F_M^2\bigr]
 =\sum_{M\in\Pot_{fin}(\N)}|M\cap[m]|^2\, \E\bigl[F_M^2\bigr].
\end{align*}
By (iii) the left hand side here is bounded in $m\in\N$. On the other hand, by the monotone convergence theorem, as $m\to\infty$, the right hand side converges to
\[\sum_{M\in\Pot_{fin}(\N)}|M|^2\, \E\bigl[F_M^2\bigr]=\sum_{p=1}^\infty p^2\,\E\bigl[(F^{(p)})^2\bigr],\]
which then must be finite as well. This proves (iv).\smallskip\\
(iv)$\Rightarrow$(i): By Proposition \ref{domD}, (iv) implies that $F\in\dom(D)$ and we thus need to make sure that $DF\in\dom(\delta)$, i.e. that there is a constant $C\in[0,\infty)$ with the property that 
\begin{align*}
\Babs{\langle DF,DG\rangle_{L^2(\kappa\otimes\P_\F)}}=\Bbabs{\sum_{k=1}^\infty \E\bigl[D_kF D_kG\bigr]}\leq C\norm{G}_2
\end{align*}
holds for any $G\in\dom(D)$. Let 
\[C:=\Biggl(\sum_{M\in \Pot_{fin}(\N)}|M|^2\,\E[F_{M}^2]\Biggr)^{1/2}\]
which is finite by (iv). 
Now, let $G\in \dom(D)$ be given with infinite Hoeffding decomposition
\[G=\sum_{M\in\Pot_{fin}(\N)}G_M.\]
Then, since  $F\in\dom(D)$ by (iv) and Proposition \ref{domD}, using orthogonality, we obtain
\begin{align*}
 &\Bbabs{\sum_{k=1}^\infty \E\bigl[D_kF D_kG\bigr]}=\Bbabs{\sum_{k=1}^\infty\sum_{\substack{M,N\in\Pot_{fin}(\N):\\
 k\in M\cap N}}\E\bigl[F_M G_N\bigr]}\leq\sum_{k=1}^\infty\sum_{\substack{M\in\Pot_{fin}(\N):\\ k\in M}}\Babs{\E\bigl[F_MG_M\bigr]}\\
 &=\sum_{M\in\Pot_{fin}(\N)}|M|\babs{\E\bigl[F_MG_M\bigr]}
 \leq\sum_{M\in\Pot_{fin}(\N)} |M| \Bigl(\E\bigl[F_M^2\bigr]\Bigr)^{1/2}\Bigl(\E\bigl[G_M^2\bigr]\Bigr)^{1/2} \\
 &\leq \biggl(\sum_{M\in\Pot_{fin}(\N)} |M|^2 \E\bigl[F_M^2\bigr]\biggr)^{1/2}
 \biggl(\sum_{M\in\Pot_{fin}(\N)}  \E\bigl[G_M^2\bigr]\biggr)^{1/2}=C\norm{G}_2,
\end{align*}
as desired.\smallskip\\
To prove the additional statement about the representation of $L$ in terms of the infinite Hoeffding decomposition, observe that
\[H:=-\sum_{M\in\Pot_{fin}(\N)}|M|\, F_M =-\sum_{p=1}^\infty p\,F^{(p)}\]
converges (unconditionally) in $L^2(\P)$ by (iv). On the other hand, we know that $-\sum_{k=1}^m D_kF$ converges in $L^2(\P)$ to $LF$ as $m\to\infty$. Moreover, for fixed $m\in\N$, Proposition \ref{HoeffD} and a short computation yield that 
\begin{align*}
&\E\biggl[\Bigl(-\sum_{k=1}^m D_kF -H\Bigr)^2\biggr]=\E\Biggl[\biggl(\sum_{k=1}^m \sum_{\substack{M\in\Pot_{fin}(\N):\\ k\in M}}F_M -\sum_{M\in\Pot_{fin}(\N)}|M|\, F_M\biggr)^2\Biggr]\\
&=\sum_{M\in\Pot_{fin}(\N)}\babs{M\setminus[m]}^2\,\E\bigl[F_M^2\bigr]\,,
\end{align*}
and the right hand side converges to $0$ as $m\to\infty$ by (iv) and the dominated convergence theorem. Hence, $LF=H$.
\end{proof}

\begin{remark}\label{domrem}
\begin{enumerate}[(a)]
\item Propositions \ref{domD}, \ref{domdelta} and \ref{domL} explicitly characterize the domains of the three Malliavin operators $D,\delta$ and $L$ and, hence, significantly extend the theory of \cite[Section 2]{DH}. We remark further that the description of $\dom(L)$ given in \cite[Definition 2.10]{DH} in general is incorrect as, by Proposition \ref{domD}, it actually describes $\dom(D)$.
\item From Propositions \ref{domD} and \ref{domL} it immediately follows that $\dom(L)\subseteq\dom(D)$. Moreover, it is clear from Proposition \ref{domL} that $F\in\dom(L)$, whenever $F\in\bigoplus_{p=0}^n\mathcal{H}_p$ for some $n\in\N$. In particular, if there is an $n\in\N$ such that $X_j$ is $\P$-a.s. constant for each $j\geq n+1$, then we have $\dom(D)=\dom(L)=L^2_\X$ and $\dom(\delta)=L^2(\kappa\otimes\P_\F)$. Recall that this is the case in the situation of functionals of only finitely many independent random variables.
\item Note that the characterizations in parts (iv) of Propositions \ref{domD} and \ref{domL} differ from those in terms of the chaotic decomposition in the Gaussian \cite{Nualart}, Poisson \cite{Lastsa} or Rademacher setting \cite{Privault} only in that a factor of $p!$ is missing from the sums. The reason for this is that, whereas the Hoeffding decomposition is given by a sum over finite subsets of $\N$, the multiple Wiener-It\^{o} integrals in these three cases involve permutations of the distinct arguments of the kernels. Apart from this conventional difference, the criteria are completely analogous to those in these three standard cases.  
\item A (possibly non-symmetric and non-homogeneous) Rademacher sequence is a sequence $\X=(X_n)_{n\in\N}$ of independent and $\{-1,1\}$-valued random variables such that $p_k:=\P(X_k=1)=1-q_k:=1-\P(X_k=-1)\in(0,1)$ for each $k\in\N$. For Rademacher sequences, the Malliavin operators $\hat{D},\hat{\delta}$ and $\hat{L}$ and the corresponding calculus have existed for several years (see e.g. \cite{Privault} for their definition and for a comprehensive treatment of the corresponding theory). As has been observed in the introduction of \cite{DH} (for the case of symmetric Rademacher sequences, i.e. $p_k=q_k=1/2$ for all $k\in\N$), the Malliavin derivative $D$ defined above does not reduce to the usual Malliavin $\hat{D}$ for a Rademacher sequence $\X$. We remark here that, actually, the derivatives $D_k$ and $\hat{D}_k$ are yet very similar. Indeed, a simple computation along the formula for $\hat{D}_k$ given in \cite[Proposition 7.3]{Privault} shows that, in fact, $D_k F=Y_k\hat{D}_kF$ for any square integrable functional $F$ of $\X$. Here, $Y_k=\frac{X_k+q_k-p_k}{2\sqrt{p_kq_k}}$ denotes the normalization of $X_k$, $k\in\N$. In view of Proposition \ref{domD} and \cite[Lemma 2.3]{KRT1}, and since $\hat{D}_kF$ and $Y_k$ are independent and $\E[Y_k^2]=1$, this in particular implies that $\dom(D)=\dom(\hat{D})$. Moreover, the respective Ornstein-Uhlenbeck generators $L$ and $\hat{L}$ are actually exactly the same (compare \cite[Proposition 10.1]{Privault} and Proposition \ref{domL} and note the different sign convention that $\hat{L}:=\hat{\delta}\hat{D}$ in \cite{Privault}). 
The similarity between $D_k$ and $\hat{D}_k$ is also reflected in the respective Stroock type formulas (see e.g. \cite[Remark 2.1]{KRT1} for the Rademacher case). Indeed, whereas one has 
$\frac{1}{p!}\E[\hat{D}_{i_1}\ldots \hat{D}_{i_p}F]=f_p(i_1,\dotsc,i_p)$ with $f_p$ the $p$-th symmetric kernel in the (Rademacher) chaotic decomposition $F=\E[F]+\sum_{p=1}^\infty \hat{J}_p(f_p)$, the formula in Corollary \ref{stroock} gives $\frac{1}{p!}\E[D_{i_1}\ldots D_{i_p}F\,|\,X_{i_1},\dotsc,X_{i_p}]=f_p(i_1,\dotsc,i_p)\prod_{l=1}^pY_{i_l}$. These strong similarities with the well-established Rademacher setting sustain our viewpoint that the derivative $D$ introduced in \cite{DH, Duer} is the natural starting point for a Malliavin structure on general product spaces. 
\item Similarly, the new formulae given in Propositions \ref{HoeffD} and \ref{domD} for the action of the $D_k$ and of $L$ in terms of the infinite Hoeffding decomposition as well as the Stroock formula in Corollary \ref{stroock} and the characterizations of $\dom(D)$ and $\dom(L)$ emphasize our point of view that the infinite Hoeffding decomposition from Proposition \ref{infhoeff} is the natural counterpart to the Wiener-It\^{o} chaos expansion for Gaussian, Poisson and Rademacher functionals.  
\end{enumerate}
\end{remark}

From the formula $LF=-\sum_{k=1}^\infty D_kF$ and the symmetry of the $D_k$ one immediately obtains that $L$ is \textit{symmetric}, i.e. that
\[\E\bigl[FLG\bigr]=\E\bigl[GLF\bigr],\quad F,G\in\dom(L).\]
We observe next that $L$ is in fact \textit{self-adjoint}, i.e. that $L=L^*$ with $L^*$ the Hilbert space adjoint of the generally unbounded operator $L$ and completely describe its spectrum. Recall that the \textit{spectrum} $\sigma(L)$ of $L$ is the set of all $\lambda\in\R$ such that the operator $L-\lambda \Id:\dom(L)\rightarrow L^2_\X$ is not bijective and that its \textit{point spectrum} $\sigma_p(L)$ consists of those $\lambda\in\sigma(L)$ such that $L-\lambda \Id:\dom(L)\rightarrow L^2_\X$ is not injective. The elements of $\sigma_p(L)$ are called \textit{eigenvalues} of $L$ and the linear subspace $\ker(L-\lambda\Id)\not=\{0\}$ is the 
\textit{eigenspace} corresponding to $\lambda\in\sigma_p(L)$. 

\begin{prop}\label{saL}
The operator $L:\dom(L)\rightarrow L^2_\X$ is self-adjoint, $-L$ is positive and one has $\sigma(L)=\sigma_p(L)\subseteq-\N_0=\{p\in\Z\,:\,p\leq 0\}$. Moreover, the eigenspace of $L$ corresponding to an eigenvalue $-p\in\sigma_p(L)$ is precisely given by the $p$-th Hoeffding space $\mathcal{H}_p$. 
\end{prop}

\begin{proof}
 Since $\mathcal{S}\subseteq\dom(L)$, $L$ is densely defined and, hence, the symmetry of $L$ and general functional analytic facts imply that $L\subseteq L^*$, that is $L^*$ is an extension of $L$. Regarding the first claim, it thus suffices to show that $\dom(L^*)\subseteq\dom(L)$. 
 Suppose on the contrary that $G\in L^2_\X\setminus\dom(L)$ and denote by $G=\sum_{p=0}^\infty G^{(p)}$ its infinite Hoeffding decomposition. Since $G\notin\dom(L)$, by Proposition \ref{domL} it holds that $\sum_{p=1}^\infty p^2\Var\bigl(G^{(p)}\bigr)=+\infty$. In particular, for any $n\in\N$, we may find $k_n\in\N$ such that 
 \[\sum_{p=1}^{k_n} p^2\Var\bigl(G^{(p)}\bigr)>n.\]
 Define
 \[F_n:=\sum_{p=1}^{k_n} p\,G^{(p)},\quad n\in\N.\]
Since $\sum_{p=1}^{k_n} p^4\Var\bigl(G^{(p)}\bigr)<\infty$,  by Proposition \ref{domL} we have $F_n\in\dom(L)$ for each $n\in\N$. However, from $LF_n=-\sum_{p=1}^{k_n} p^2\,G^{(p)}$ we obtain that
 \begin{align*}
  \babs{\langle LF_n,G\rangle}&=\sum_{p=1}^{k_n}p^2\,\Var\bigl(G^{(p)}\bigr)
  =\norm{F_n}_2^2>\sqrt{n}\norm{F_n}_2,\quad n\in\N.
 \end{align*}
Therefore, it follows from the general definition of $\dom(L^*)$ that $G\notin \dom(L^*)$. Thus, $\dom(L^*)\subseteq\dom(L)$ and, hence, $L=L^*$ as claimed. The positivity of $-L$ follows from the fact that for $F\in\dom(L)$ with infinite Hoeffding decomposition \eqref{genhd2} one has 
\begin{align*}
 \E\bigl[F(-LF)\bigr]=\sum_{p=0}^\infty p\, \E\Bigl[\bigl(F^{(p)}\bigr)^2\Bigr]\geq0,
\end{align*}
where the orthogonality of the $F^{(p)}$ and $L^2(\P)$-convergence have been used.
\smallskip\\
We next show that $\sigma_p(L)\subseteq-\N_0$. By Proposition \ref{domL} we have that $LF=-pF$ holds for any $F\in\mathcal{H}_p$, $p\in\N_0$. Conversely, suppose that $G\in \dom(L)$ satisfies $LG=\lambda G$ for some $\lambda\in\R$ and let $G= \sum_{p=0}^\infty G^{(p)}$ denote its infinite Hoeffding decomposition. Then, from Proposition \ref{domL} we have 
\[\sum_{p=0}^\infty \lambda\,G^{(p)}=\lambda G=LG=-\sum_{p=0}^\infty p\, G^{(p)}\]
and the uniqueness of the Hoeffding decomposition for $LG$ implies that 
$\lambda\,G^{(p)}=-p\,G^{(p)}$ holds $\P$-a.s. for any $p\in\N_0$. This implies that $\Var(G^{(p)})>0$ for at most one $p\in\N_0$ and if this $p$ exists, then $\lambda=-p$. Hence, we have shown that $\sigma_p(L)\subseteq-\N_0$ and that for $-p\in\sigma_p(L)$ the corresponding eigenspace is given by $\mathcal{H}_p$. \smallskip\\
We finally show that $\sigma(L)=\sigma_p(L)$. To this end, it suffices to prove that, for any $\lambda\in\R\setminus\sigma_p(L)$ the operator 
$T_\lambda:=L-\lambda\Id:\dom(L)\rightarrow L^2_\X$ is bijective and that there is a $c_\lambda>0$ such that $\norm{T_\lambda F}_2\geq c_\lambda\norm{F}_2$ for any $F\in\dom(L)$. We have already proved that $T_\lambda$ is injective, so let $G\in L^2_\X$ with the infinite Hoeffding decomposition $G= \sum_{p=0}^\infty G^{(p)}=\sum_{p\in \sigma_p(L)} G^{(p)}$ be given. Note that the second representation for $G$ holds, since $G^{(p)}=0$ $\P$-a.s. for any $p\in-\N_0\setminus\sigma_p(L)$ as $\mathcal{H}_p=\{0\}$ in this case. 
Then, 
\[F:=-\sum_{p\in \sigma_p(L)} \frac{1}{\lambda+p} G^{(p)}\]
is well-defined and contained in $\dom(L)$ by Proposition \ref{domL} and one has $T_\lambda F=G$. Moreover, for any $F\in\dom(L)$  with infinite Hoeffding decomposition $F= \sum_{p\in \sigma_p(L)}  F^{(p)}$ one has 
\begin{align*}
 \norm{T_\lambda F}_2^2=\sum_{p\in \sigma_p(L)} (\lambda+p)^2 \norm{F^{(p)}}_2^2
 \geq d\bigl(\lambda,\sigma_p(L)\bigr)^2\norm{F}_2^2,
\end{align*}
where $d(\lambda,\sigma_p(L)):=\min\{|\lambda-q|\,:\,q\in\sigma_p(L)\}>0$ since $\lambda\notin\sigma_p(L)$ and $\sigma_p(L)\subseteq-\N_0$ . Thus, we can take $c_\lambda=d(\lambda,\sigma_p(L))$.
\end{proof}

\begin{remark}\label{saLrem}
\begin{enumerate}[(a)]
\item In the situation of Proposition \ref{saL} it is not always true that every $-p\in-\N_0$ is an eigenvalue of $L$. For instance, if for some $n\in\N$, the random variables $X_{n+1},X_{n+2},\dotsc$ are $\P$-a.s.constant, then one necessarily has $\mathcal{H}_p=\{0\}$ for each $p\geq n+1$ due to the classical Hoeffding decomposition for $L^2(\P)$-functionals of $X_1,\dotsc,X_n$.
\item Since $\E[LF]=0$ for any $F\in\dom(L)$, no random variable with non-zero mean is contained in the image $\im(L)$ of $L$. Suppose, on the contrary, that $G\in L^2_\X$ satisfies $\E[G]=0$. Then, $G$ has infinite Hoeffding decomposition of the form $G=\sum_{p=1}^\infty G^{(p)}=\sum_{ M\in\Pot_{fin}(\N)}G_M $ and it follows from Proposition \ref{domL} that 
\[F:=-\sum_{p=1}^\infty \frac{1}{p}G^{(p)}=-\sum_{\emptyset\not= M\in\Pot_{fin}(\N)}\frac{1}{|M|}G_M \in\dom(L)\]
 and $LF=G$. Moreover since, by Proposition \ref{saL}, $\ker(L)$ only consists of the constants, $F$ is the only centered random variable in $\dom(L)$ with this property. As usual, we write $F=L^{-1}G$ and call $L^{-1}$ the \textit{pseudo-inverse} of the Ornstein-Uhlenbeck generator $L$. Note that, with this definition, one has 
\begin{itemize}
\item $L L^{-1}G=G$ for each centered $G\in L^2_\X$.
\item $L^{-1} LF=F-\E[F]$ for any $F\in\dom(L)$.
\end{itemize}
\item The Ornstein-Uhlenbeck generator $L$ and its pseudo-inverse $L^{-1}$ have also been analyzed in \cite[Lemma 2.4]{Duer}. There, it has been provided that $L$ is \textit{essentially} self-adjoint, positive, has pure point spectrum $-\N$ (which, in view of part (a) of this remark, is not completely correct) and that its image is dense in the subspace $\{1\}^\perp$ of $L^2_\X$ of centered random variables. Note that our Proposition \ref{saL} strengthens some of these findings, by establishing that $L$ is in fact genuinely self-adjoint and we have moreover just seen that the image of $L$ actually coincides with $\{1\}^\perp$.
\end{enumerate}
\end{remark}

\subsection{Ornstein-Uhlenbeck semigroup and Mehler formula}\label{OU}
Thanks to the infinite Hoeffding decomposition, we are able to define the Ornstein-Uhlenbeck semigroup in a similar way as in the Gaussian \cite{Nualart}, Poisson \cite{Lastsa} and Rademacher \cite{Privault} settings. Thus, for $F\in L^2_\X$ with inifinite Hoeffding decomposition \eqref{genhd2} and $t\in[0,\infty)$ we define
\begin{equation*}
P_tF:=\sum_{M\in\Pot_{fin}(\N)}e^{-t|M|} F_M=\sum_{p=0}^\infty e^{-tp} J_p(F),
\end{equation*} 
which is obviously again an element of $L^2_\X$. Then, it is easy to see that the family $(P_t)_{t\geq0}$ of operators has the semigroup property, i.e. $P_0=\Id$ and $P_tP_s=P_{t+s}$, $s,t\geq0$, and we call it the \textit{Ornstein-Uhlenbeck semigroup} on $L^2_\X$. Note that, by Proposition \ref{HoeffD}, the operators $D_k$ and $P_t$ commute, that is, for $F\in L^2_\X$, $t\geq0$ and $k\in\N$ one has 
\begin{equation}\label{commrel}
D_kP_tF= \sum_{M\in\Pot_{fin}(\N):k\in M}e^{-t|M|} F_M=P_tD_kF
\end{equation}
and, if $F$ is moreover centered, one further has
\begin{equation}\label{commrel2}
D_kL^{-1}F= -D_k\sum_{\emptyset\not=M\in\Pot_{fin}(\N)}\frac{1}{|M|} F_M=-\sum_{\emptyset\not=M\in\Pot_{fin}(\N) :k\in M}\frac{1}{|M|} F_M=L^{-1} D_kF.
\end{equation}

\begin{remark}\label{commrelrem}
\begin{enumerate}[(a)]
\item The commutation relation \eqref{commrel} essentially differs from those that hold in the Gaussian, Poisson \cite{LPS} and Rademacher situations \cite{KRT2}, respectively, in that a factor of $e^{-t}$ is missing from the right hand side of \eqref{commrel}. Incidentally, in \cite[Theorem 2.13]{DH} it is incorrectly stated that $D_kP_tF=e^{-t}P_tD_kF$ also holds in the present situation.  
\item The operator $L$ is indeed the infinitesimal generator of the semigroup $(P_t)_{t\geq0}$. To see this, let $F\in \dom(L)$ with inifinite Hoeffding decomposition \eqref{genhd2} be fixed and observe that with $F_t:=t^{-1}(P_tF-F)$, $t>0$, one has
\begin{align*}
\norm{F_t-LF}_2^2&=\sum_{p=0}^\infty\biggl(\frac{e^{-tp}-1}{t}+p\biggr)^2\E\bigl[J_p(F)^2\bigr]\stackrel{t\downarrow0}{\longrightarrow}0
\end{align*}
by the dominated convergence theorem, which applies thanks to Proposition \ref{domL} as the square of the bracket in the above sum is bounded by $p^2$ for all $t>0$.
\item In \cite{DH} the semigroup $(P_t)_{t\geq0}$ has been introduced in a different way by first considering functionals of finitely many independent random variables, in which case a suitable Markov process with infinitesimal generator $L$ may be explicitly constructed via a Gibbs sampling procedure (see also the discussion in the beginning of Subsection 2.1 in \cite{D23}), and then invoking the Hille-Yosida-Theorem for the general case. We emphasize that our definition of $(P_t)_{t\geq0}$ via the infinite Hoeffding decomposition is more direct and also more in the spirit of the corresponding usual definitions in the situation of Gaussian \cite[Definition 1.4.1]{Nualart}, Poisson \cite[Section 7]{Lastsa} and Rademacher functionals \cite[Section 10]{Privault}.
\end{enumerate}
\end{remark}

The following result, the \textit{Mehler formula}, gives a very useful alternative representation for the semigroup $(P_t)_{t\geq0}$ as the transition semigroup of the Markov process $(\X(t))_{t\geq0}$ that has been defined in Subsection \ref{setup}. By Remark \ref{commrelrem} (b), this further implies that the operator $L$ is the infinitesimal generator of $(\X(t))_{t\geq0}$, which, in view of Proposition \ref{saL} is thus reversible. In a slightly different form, this formula has also been stated in \cite[Theorem 2.12]{DH} and a proof has been outlined. We include a neat probabilistic proof for the sake of completeness. Note that analogous results hold for Gaussian \cite{Nualart}, Poisson \cite{LPS} and Rademacher functionals \cite{KRT2}.

\begin{prop}[Mehler formula]\label{Mehler}
For any $t\in[0,\infty)$ and $F\in L^2_\X$ with representative $f\in L^2(\mu)$ one has 
\[P_tF=\E\Bigl[f\bigl(\X(t)\bigr)\,\Big|\,\X\Bigr].\]
\end{prop} 

\begin{proof}
Let $t\geq0$ and let $F$ again have the inifinite Hoeffding decomposition \eqref{genhd2}. Then, by the factorization lemma, for each $M\in\Pot_{fin}(\N)$, there is a measurable function $f_M$ such that 
$F_M=f_M(X_j,j\in M)$. For $M\in\Pot_{fin}(\N)$ and $\eps=(\eps_j)_{j\in M}\in\{0,1\}^M$ let $\X^{M,\eps}:=(X^{M,\eps}_j)_{j\in M}$ be given by 
\[X^{M,\eps}_j:=\begin{cases}
X_j',&\eps_j=0\\
X_j,&\eps_j=1.
\end{cases}\]
Then, by $L^2(\P)$-convergence and independence, we have 
\begin{align*}
&\E\Bigl[f\bigl(\X(t)\bigr)\,\Big|\,\X\Bigr]=\sum_{M\in\Pot_{fin}(\N)}\E\Bigl[f_M\bigl(\X_j(t)\bigr)\,\Big|\,\X\Bigr]\\
&=\sum_{M\in\Pot_{fin}(\N)}\sum_{\eps=(\eps_j)_{j\in M}\in\{0,1\}^M}\E\biggl[\prod_{j\in M}\Bigl(\1_{\{Z_j>t\}}\eps_j+\1_{\{Z_j\leq t\}}(1-\eps_j)\Bigr)f_M\bigl(\X^{M,\eps}\bigr)\,\bigg|\,\X\biggr]\\
&=\sum_{M\in\Pot_{fin}(\N)}\sum_{\eps=(\eps_j)_{j\in M}\in\{0,1\}^M}\prod_{j\in M}\Bigl(e^{-t}\eps_j+(1-e^{-t})(1-\eps_j)\Bigr) \E\Bigl[f_M\bigl(\X^{M,\eps}\bigr)\,\Big|\,\X\Bigr].
\end{align*}
Now note that by the degeneracy of $F_M$ and the independence of $\X$ and $\X'$ we have 
\begin{equation*}
\E\Bigl[f_M\bigl(\X^{M,\eps}\bigr)\,\Big|\,\X\Bigr]=\begin{cases}
F_M,& \text{if }\eps_j=1\text{ for all } j\in M,\\
0,&\text{otherwise}.
\end{cases}
\end{equation*}
Thus, we conlcude that 
\[\E\Bigl[f\bigl(\X(t)\bigr)\,\Big|\,\X\Bigr]=\sum_{M\in\Pot_{fin}(\N)} F_M\prod_{j\in M} e^{-t}=\sum_{M\in\Pot_{fin}(\N)}e^{-t|M|}  F_M=P_tF,\]
as claimed. 
\end{proof}

The pseudo-inverse $L^{-1}$ of $L$ is in general difficult to compute or to keep track of, when the infinite Hoeffding decomposition is not available. For such cases, the following result, in combination with Proposition \ref{Mehler}, turns out to be very useful. Analogous formulas hold on the Poisson space \cite[Theorem 3.2]{LPS} and for Rademacher functionals \cite{KRT2}.

\begin{lemma}\label{pilemma}
Let $F\in L^2_\X$ be centered. Then, 
\[-L^{-1}F=\int_0^\infty P_tF dt,\]
where the right hand side denotes a Bochner integral with values in $L^2(\P)$.
\end{lemma}

\begin{proof}
Let $F$ have the inifinite Hoeffding decomposition \eqref{genhd2}. We first show that the integral on the right hand side is a well-defined random variable in $L^2(\P)$. By Bochner's integrability criterion this follows from
\begin{align*}
&\int_0^\infty \norm{P_tF}_2 dt=\int_0^\infty \Biggl(\E\biggl[\Bigl(\sum_{\substack{ M\in\Pot_{fin}(\N): M\not=\emptyset}} e^{-t|M|} F_M\Bigr)^2\biggr]\Biggr)^{1/2}dt\\
&=\int_0^\infty \Biggl(\sum_{\substack{ M\in\Pot_{fin}(\N): M\not=\emptyset}} e^{-2t|M|}\E\bigl[F_M^2\bigr]\Biggr)^{1/2}dt
\leq \int_0^\infty e^{-t}\Biggl(\sum_{\substack{ M\in\Pot_{fin}(\N): M\not=\emptyset}} \E\bigl[F_M^2\bigr]\Biggr)^{1/2}dt\\
&=\norm{F}_2\int_0^\infty e^{-t}dt=\norm{F}_2<\infty,
\end{align*} 
where we have used that $F_\emptyset=\E[F]=0$. Next, by the definition of $L^{-1}$ we have
\begin{align*}
L^{-1}F=-\sum_{\emptyset\not= M\in\Pot_{fin}(\N)}\frac{1}{|M|} F_M.
\end{align*}
 Moreover, for $n\in\N$, again since $F_\emptyset=\E[F]=0$, we have
\begin{align*}
\int_0^\infty \sum_{ M\in\Pot([n])}e^{-t|M|} F_Mdt=\sum_{\emptyset\not= M\in\Pot([n])}\frac{1}{|M|} F_M
\end{align*}  
and the right hand side converges in $L^2(\P)$ to $-L^{-1} F$ as $n\to\infty$. Thus, it remains to show that
\[\BEnorm{ \int_0^\infty \biggl(P_tF-\sum_{ M\in\Pot([n])}e^{-t|M|} F_M\Bigr)dt}\stackrel{n\to\infty}{\longrightarrow}0.\]
Indeed, similarly as above, using the triangle inequality for Bochner integrals we obtain
\begin{align*}
&\BEnorm{ \int_0^\infty \biggl(P_tF-\sum_{ M\in\Pot([n])}e^{-t|M|} F_M\Bigr)dt}\leq \int_0^\infty \BEnorm{P_tF-\sum_{ M\in\Pot([n])}e^{-t|M|} F_M}dt\\
&=\int_0^\infty \BEnorm{\sum_{\substack{ M\in\Pot_{fin}(\N):\\ M\cap[n+1,\infty)\not=\emptyset}}e^{-t|M|} F_M}dt
=\int_0^\infty \Biggl(\sum_{\substack{ M\in\Pot_{fin}(\N):\\ M\cap[n+1,\infty)\not=\emptyset}} e^{-2t|M|}\E\bigl[F_M^2\bigr]\Biggr)^{1/2}dt \\
&\leq \int_0^\infty e^{-t}\Biggl(\sum_{\substack{ M\in\Pot_{fin}(\N):\\ M\cap[n+1,\infty)\not=\emptyset}} \E\bigl[F_M^2\bigr]\Biggr)^{1/2}dt
=\Biggl(\sum_{\substack{ M\in\Pot_{fin}(\N):\\ M\cap[n+1,\infty)\not=\emptyset}} \E\bigl[F_M^2\bigr]\Biggr)^{1/2},
\end{align*}
which converges to zero as $n\to\infty$ by the dominated convergence theorem as $F\in L^2_\X$.

\end{proof}

To prepare the proof of the next result, suppose that $G\in L^2_\X$ with representative $g\in L^2(\mu)$ is a random variable whose infinite Hoeffding decomposition is of the form 
\begin{equation}\label{HDG}
G=\sum_{\substack{M\in \Pot_{fin}(\N):\\k_1,\dotsc,k_d\in M}}G_M
\end{equation}
for some fixed $d\in\N$ and pairwise distinct, fixed $k_1,\dotsc,k_d\in\N$. Then, by the factorization lemma, there are measurable functions $g_M$ such that $G_M=g_M(X_j,j\in M)$ for all $M\in \Pot_{fin}(\N)$ with $k_1,\dotsc,k_d\in M$. Now, for primarily fixed $y_1\in E_{k_1},\dotsc,y_d\in E_{k_d}$, we define 
\begin{align}\label{Gtilde}
&\tilde{G}(y_1,\dotsc,y_d):=g\bigl(y_1,\dotsc,y_d,(X_j)_{j\in\N\setminus\{k_1,\dotsc,k_d\}}\bigr)\notag\\
&=\sum_{\substack{M\in \Pot_{fin}(\N):\\k_1,\dotsc,k_d\in M}}g_{M}\bigl(y_1,\dotsc,y_d,(X_j)_{j\in M\setminus\{k_1,\dotsc,k_d\}}\bigr)\notag\\
&=\sum_{\substack{N\in \Pot_{fin}(\N):\\k_1,\dotsc,k_d\notin N}}g_{N\cup\{k_1,\dotsc,k_d\}}\bigl(y_1,\dotsc,y_d,(X_j)_{j\in N}\bigr)
\end{align} 
in such way that $\tilde{G}(X_{k_1},\dotsc,X_{k_d})=G$. Note further that, due to the degeneracy of the $G_M$, the right hand side of \eqref{Gtilde} coincides with the infinite Hoeffding decomposition of $\tilde{G}(y_1,\dotsc,y_d)$.  
Hence, by definition of $P_t$,
\begin{align*}
h_t(y_1,\dotsc,y_d)&:=
P_t\bigl(\tilde{G}(y_1,\dotsc,y_d)\bigr)=\sum_{\substack{N\in \Pot_{fin}(\N):\\k_1,\dotsc,k_d\notin N}} e^{-t|N|}g_{N\cup\{k_1,\dotsc,k_d\}}\bigl(y_1,\dotsc,y_d,(X_j)_{j\in N}\bigr)\\
&=\sum_{\substack{M\in \Pot_{fin}(\N):\\k_1,\dotsc,k_d\in M}}e^{-t(|M|-d)}g_{M}\bigl(y_1,\dotsc,y_d,(X_j)_{j\in M\setminus\{k_1,\dotsc,k_d\}}\bigr)
\end{align*}
and, thus,
\[h_t(X_{k_1},\dotsc,X_{k_d})= \sum_{\substack{M\in \Pot_{fin}(\N):\\k_1,\dotsc,k_d\in M}}e^{-t(|M|-d)}g_{M}\bigl(X_j,j\in M\bigr)=e^{td} P_tG\] 
and
\begin{equation}\label{mf1}
P_tG=e^{-td}h_t(X_{k_1},\dotsc,X_{k_d}).
\end{equation}
We stress that both $h_t$ and $\tilde{G}(y_1,\dotsc,y_d)$ implicitly depend on $d$ and $k_1,\dotsc,k_d$ and that we may therefore more precisely denote these quantities by $\tilde{G}_{k_1,\dotsc,k_d}(y_1,\dotsc,y_d)$ and $h_{(k_1,\dotsc,k_d,t)}$, respectively.

Note that, with the definition
\[\tilde{g}_{(y_1,\dotsc,y_d)}\bigl((x_j)_{j\in\N}\bigr):=g\bigl(y_1,\dotsc,y_d,(x_j)_{j\in\N\setminus\{k_1,\dotsc,k_d\}}\bigr),\quad (x_j)_{j\in\N}\in E=\prod_{j\in\N}E_j,\]
we can write 
\begin{equation*}\label{mehlerfix}
\tilde{G}(y_1,\dotsc,y_d)=\tilde{g}_{(y_1,\dotsc,y_d)}(\X)
\end{equation*}
and, hence, by Proposition \ref{Mehler}, we also have
\begin{equation}\label{mehlerfix}
P_t\bigl(\tilde{G}(y_1,\dotsc,y_d)\bigr)=\E\Bigl[\tilde{g}_{(y_1,\dotsc,y_d)}\bigl(\X(t)\bigr)\,\Big|\,\X\Bigr].
\end{equation}
\smallskip\\

In what follows, for $d\in\N$ and pairwise distinct $k_1,\dotsc,k_d\in\N$, we denote by
\[D^d_{k_1,\dotsc,k_d}:=D_{k_1}\circ D_{k_2}\circ\ldots\circ D_{k_d}\]
the \textit{iterated Malliavin derivative}. Note that the order of the elements $k_1,\dotsc,k_d$ is immaterial, here. 
The following result is useful for bounding certain expressions arising from abstract normal approximation bounds.
\begin{prop}\label{Lmin}
Suppose that, for some $r\in[1,\infty)$, $F\in L^2_\X\cap L^r_\X$. Moreover, let $d\in\N$ and $k_1,\dotsc,k_d\in\N$ be pairwise distinct. Then, $D^d_{k_1,\dotsc,k_d}L^{-1}F\in L^r_\X$ and  $\norm{D^d_{k_1,\dotsc,k_d}L^{-1}F}_r\leq d^{-1}\norm{D^d_{k_1,\dotsc,k_d}F}_r$.
\end{prop}

\begin{proof}
We use the notation introduced in the last paragraph. Let $F$ have the infinite Hoeffding decomposition \eqref{genhd2}. We first observe that, by Proposition \ref{HoeffD}, the infinite Hoeffding decomposition of
\[G:= D^d_{k_1,\dotsc,k_d}F=\sum_{\substack{M\in \Pot_{fin}(\N):\\k_1,\dotsc,k_d\in M}}F_M\]
is of the form \eqref{HDG}. Moreover, 
from Lemma \ref{pilemma}, \eqref{commrel2} and \eqref{mf1} we have that 
\begin{align}\label{mf2}
D^d_{k_1,\dotsc,k_d}L^{-1}F&=L^{-1} D^d_{k_1,\dotsc,k_d}F =-\int_0^\infty P_t D^d_{k_1,\dotsc,k_d}F dt=-\int_0^\infty P_t G dt\notag\\
& =- \int_0^\infty e^{-td}h_t(X_{k_1},\dotsc,X_{k_d})dt.
\end{align}
Hence, from \eqref{mf2}, using Fubini's theorem and Jensen's inequality twice, we obtain that
\begin{align}\label{mf3}
&\E\babs{D^d_{k_1,\dotsc,k_d}L^{-1}F}^r=d^{-r}\E\Bbabs{\int_0^\infty h_t(X_{k_1},\dotsc,X_{k_d})d e^{-td}dt }^r\notag\\
&\leq d^{-r}\E\Biggl[\int_0^\infty \babs{h_t(X_{k_1},\dotsc,X_{k_d})}^rd e^{-td}dt\Biggr]\notag\\
&=d^{-r}\int_0^\infty\E\Bbabs{\int_{E_{k_1}\times\ldots\times E_{k_d}}P_t\bigl(\tilde{G}(y_1,\dotsc,y_d)\bigr)d\bigotimes_{l=1}^d \mu_{k_l}(y_1,\dotsc,y_d) }^r d e^{-td}dt\notag\\
&\leq d^{-r}\int_0^\infty \int_{E_{k_1}\times\ldots\times E_{k_d}}\E\Babs{ P_t\bigl(\tilde{G}(y_1,\dotsc,y_d)\bigr)}^r d\bigotimes_{l=1}^d \mu_{k_l}(y_1,\dotsc,y_d)d e^{-td}dt.
\end{align}
Now, by \eqref{mehlerfix}, Proposition \ref{Mehler}, the conditional Jensen inequality and since $\X(t)$ has the same distribution as $\X$, for fixed $t\in[0,\infty), y_1\in E_{k_1},\dotsc,y_d\in E_{k_d}$, we have 
\begin{align*}
\E\Babs{ P_t\bigl(\tilde{G}(y_1,\dotsc,y_d)\bigr)}^r& =\E\Babs{\E\Bigl[\tilde{g}_{(y_1,\dotsc,y_d)}\bigl(\X(t)\bigr)\,\Big|\,\X\Bigr]}^r\leq \E\biggl[ \E\Bigl[\babs{\tilde{g}_{(y_1,\dotsc,y_d)}\bigl(\X(t)\bigr)}^r\,\Big|\,\X\Bigr]\biggr]\\
&= \E\babs{\tilde{g}_{(y_1,\dotsc,y_d)}(\X)}^r=\E\Babs{g\bigl((X_j)_{j\in\N\setminus\{k_1,\dotsc,k_d\}},y_1,\dotsc,y_d\bigr)}^r
\end{align*}
so that \eqref{mf3} and Fubini's theorem imply that 
\begin{align*}
&\E\babs{D^d_{k_1,\dotsc,k_d}L^{-1}F}^r\\
&\leq d^{-r}\int_0^\infty \int_{E_{k_1}\times\ldots\times E_{k_d}}\E\Babs{g\bigl((X_j)_{j\in\N\setminus\{k_1,\dotsc,k_d\}},y_1,\dotsc,y_d\bigr)}^r d\bigotimes_{l=1}^d \mu_{k_l}(y_1,\dotsc,y_d)d e^{-td}dt\\
&= d^{-r}\int_0^\infty \E\babs{G}^r d e^{-td}dt=d^{-r}\,\E\babs{ D^d_{k_1,\dotsc,k_d}F}^r,
\end{align*}
as claimed.
\end{proof}

\begin{remark}
Similar moment bounds as in Proposition \ref{Lmin} have been proved in the Poisson \cite[Lemma 3.4]{LPS} and Rademacher situations \cite[Proposition 3.3]{KRT2}. The above proof, however, is rather different from the proofs in those settings, which is mainly due to the essential difference in the commutation relation \eqref{commrel}.
\end{remark}

\subsection{Carr\'{e}-du-champ operator}\label{cdc}
We finally introduce the \textit{carr\'{e}-du-champ operator} $\Gamma$ associated to the Markov generator $L$. We refer to the monograph \cite{BGL14} for a comprehensive study of (diffusive) Markovian generators via their associated carr\'{e}-du-champ operators. 

For $F,G\in\dom(L)$ such that also $FG\in\dom(L)$ this bilinear operator is defined via 
\begin{equation*}
 \Gamma(F,G):=\frac12\Bigl(L(FG)-GLF-FLG\Bigr)\in L^1(\P)
\end{equation*}
so that the symmetry of $L$ and the fact that $L(FG)$ is centered imply the \textit{carr\'{e}-du-champ integration by parts formula}
\begin{equation}\label{cdcintparts}
 \E\bigl[\Gamma(F,G)\bigr]=-\E\bigl[FLG\bigr]=-\E\bigl[GLF\bigr].
\end{equation}
We remark that, contrary to the situation considered in \cite{BGL14}, the operator $L$ associated to $\X$ is \textit{non-diffusive} in the sense that, for a $C^1$-function $\psi$ on $\R$,
\begin{equation*}
 R_\psi(F,G):=\Gamma(\psi(F),G)-\psi'(F)\Gamma(F,G)\not=0
\end{equation*}
in general.
For the purpose of normal approximation it will be important to control the size 
of the error term $R_\psi(F,G)$. To this end, as in \cite{DP18a} for the Poisson and in \cite{DK19} for the Rademacher case,
we provide a new alternative representation of $\Gamma(F,G)$ as well as an alternative integration by parts formula. 
For $F,G\in \dom(D)$ with respective representatives $f,g\in L^2(\mu)$ we therefore define
\begin{equation}\label{Gamma0}
\Gamma_0(F,G):=\frac{1}{2}\sum_{k=1}^\infty\E\biggl[\Bigl(f\bigl(\X^{(k)}\bigr)-f\bigl(\X\bigr)\Bigr) \Bigl(g\bigl(\X^{(k)}\bigr)-g\bigl(\X\bigr)\Bigr)\,\biggl|\,\X\biggr].
\end{equation}
Note that it follows from the estimate
\begin{align*}
&\E \Biggl[\sum_{k=1}^\infty\Bbabs{ \E\biggl[\Bigl(f\bigl(\X^{(k)}\bigr)-f\bigl(\X\bigr)\Bigr) \Bigl(g\bigl(\X^{(k)}\bigr)-g\bigl(\X\bigr)\Bigr)\,\biggl|\,\X\biggr]   }\Biggr]\\
&\leq\sum_{k=1}^\infty\E\biggl[\Babs{f\bigl(\X^{(k)}\bigr)-f\bigl(\X\bigr)}\Babs{g\bigl(\X^{(k)}\bigr)-g\bigl(\X\bigr)}\biggr]\\
&\leq\biggl(\sum_{k=1}^\infty\E\Bigl[\Bigl(f\bigl(\X^{(k)}\bigr)-f\bigl(\X\bigr)\Bigr)^2 \Bigr]\biggr)^{1/2}
\biggl(\sum_{k=1}^\infty\E\Bigl[\Bigl(g\bigl(\X^{(k)}\bigr)-g\bigl(\X\bigr)\Bigr)^2 \Bigr]\biggr)^{1/2}
\end{align*}
and from Proposition \ref{domD} that $\Gamma_0(F,G)$ is a well-defined element of $L^1(\P)$ for all $F,G\in\dom(D)$ and that also 
\begin{equation*}
\Gamma_0(F,G)=\frac{1}{2}\E\biggl[\sum_{k=1}^\infty\Bigl(f\bigl(\X^{(k)}\bigr)-f\bigl(\X\bigr)\Bigr) \Bigl(g\bigl(\X^{(k)}\bigr)-g\bigl(\X\bigr)\Bigr)\,\biggl|\,\X\biggr]
\end{equation*}
in this case. 

The following important result shows that $\Gamma_0$ coincides with $\Gamma$, whenever the latter is defined.

\begin{prop}\label{formgamma}
Let $F,G\in\dom(L)$ be such that $FG\in\dom(L)$ as well. 
Then, $\Gamma_0(F,G)\in L^1(\P)$ is well-defined and, in fact, $\Gamma(F,G)=\Gamma_0(F,G)$. 
\end{prop}

\begin{proof}
Since $\dom(L)\subseteq\dom(D)$ by Propositions \ref{domD} and \ref{domL} it is clear that $\Gamma_0(F,G)\in L^1(\P)$ is well-defined. As before, let $f,g\in L^2(\mu)$ be representatives of $F$ and $G$, respectively. By definition, we have (with convergence in $L^1_\X$)
\begin{align*}
2\Gamma(F,G)&=L(FG)-GLF-FLG=\sum_{k=1}^\infty\Bigl(-D_k(FG)+GD_kF+FD_kG\Bigr)\\
&=\sum_{k=1}^\infty\biggl(-\E\bigl[(fg)(\X)-(fg)(\X^{(k)})\,\bigl|\,\X\bigr]+g(\X)\E\bigl[f(\X)-f(\X^{(k)})\,\bigl|\,\X\bigr]\\
&\hspace{3cm}+f(\X)\E\bigl[g(\X)-g(\X^{(k)})\,\bigl|\,\X\bigr]\biggr)\\
&=\sum_{k=1}^\infty\E\biggl[\Bigl(f\bigl(\X^{(k)}\bigr)-f\bigl(\X\bigr)\Bigr) \Bigl(g\bigl(\X^{(k)}\bigr)-g\bigl(\X\bigr)\Bigr)\,\biggl|\,\X\biggr]=2\Gamma_0(F,G),
\end{align*}
where we have used the fact that $fg$ is a representative of $FG$. 
\end{proof}

\begin{prop}[Carr\'{e}-du-champ integration-by-parts]\label{intpartsgamma0}
 Suppose that $F\in\dom(D)$ and $G\in\dom(L)$. Then, 
 \[\E\bigl[FLG\bigr] = -\E\bigl[\Gamma_0(F,G)\bigr].\]
\end{prop}

\begin{proof}
Let $f,g\in L^2(\mu)$ be representatives of $F$ and $G$, respectively. Since $G\in\dom(L)$, the series $-\sum_{k=1}^\infty D_kG$ converges in $L^2(\P)$ to $LG$ and we have
 \begin{align*}
  &\E\bigl[FLG\bigr]=-\sum_{k=1}^\infty\E\bigl[FD_kG\bigr]
  =-\sum_{k=1}^\infty\E\Bigl[f(\X)\E\bigl[g(\X)-g(\X^{(k)})\,\bigl|\,\X\bigr]\Bigr]\\
  &=-\sum_{k=1}^\infty\E\Bigl[\E\bigl[f(\X)\bigl( g(\X)-g(\X^{(k)})\bigr)\,\bigl|\,\X\bigr]\Bigr]
  =-\sum_{k=1}^\infty\E\Bigl[f(\X)\bigl( g(\X)-g(\X^{(k)})\bigr)\Bigr].
 \end{align*}
Now, for each $k\in\N$, $(\X,\X^{(k)})$ has the same distribution as $(\X^{(k)},\X)$ so that we also have 
\begin{align*}
  \E\bigl[FLG\bigr]&=\sum_{k=1}^\infty\E\Bigl[f(\X^{(k)})\bigl( g(\X)-g(\X^{(k)})\bigr)\Bigr]
\end{align*}
implying
\begin{align*}
 &\E\bigl[FLG\bigr]=-\frac12\sum_{k=1}^\infty\E\Bigl[\bigl(f(\X)-f(\X^{(k)})\bigr)\bigl( g(\X)-g(\X^{(k)})\bigr)\Bigr]  \\
&= -\frac12\sum_{k=1}^\infty\E\Bigl[\E\bigl[\bigl(f(\X)-f(\X^{(k)})\bigr)\bigl( g(\X)-g(\X^{(k)})\bigr)\,\bigl|\,\X\bigr]\Bigr]  
 =-\E\bigl[\Gamma_0(F,G)\bigr].
\end{align*}
Note that for the last identity we have used the fact that $F,G\in\dom(D)$ in order to interchange the expectation and the infinite sum. 
\end{proof}

The carr\'{e}-du-champ operator $\Gamma$ in general is closely related to the associated \textit{Dirichlet form} $\mathcal{E}$ (see e.g. \cite[Section 1.7]{BGL14}). We further refer to \cite{BH} for a comprehensive treatment of (abstract and concrete) Dirichlet forms. For $F,G\in\dom(L)$ this symmetric bilinear form is usually defined by 
\begin{equation*}
\mathcal{E}(F,G):=-\E\bigl[FLG\bigr]=-\E\bigl[GLF\bigr],
\end{equation*}
where we have used the symmetry of $L$. This formula is also provided in \cite[Section 4]{DH}. If, additionally, $FG\in\dom(L)$, then by \eqref{cdcintparts} it also holds that 
\begin{equation*}
\mathcal{E}(F,G)=\E\bigl[\Gamma(F,G)\bigr].
\end{equation*} 
Now, since $\Gamma_0$ extends $\Gamma$, and thanks to Proposition \ref{intpartsgamma0} we can extend the definition of $\mathcal{E}$ by letting 
\begin{equation*}
\mathcal{E}(F,G):=\E\bigl[\Gamma_0(F,G)\bigr]=\frac{1}{2}\sum_{k=1}^\infty\E\biggl[\Bigl(f\bigl(\X^{(k)}\bigr)-f\bigl(\X\bigr)\Bigr) \Bigl(g\bigl(\X^{(k)}\bigr)-g\bigl(\X\bigr)\Bigr)\biggr]
\end{equation*} 
for all $F,G\in\dom(D)\supseteq\dom(L)$. Note that, in \cite[Definition 4.1]{DH}, the alternative definition 
\begin{equation*}
\mathcal{E}(F,G)=\langle DF,DG\rangle_{L^2(\kappa\otimes\P_\F)}=\sum_{k=1}^\infty \E\bigl[D_kF D_kG\bigr], \quad F,G\in\dom(D),
\end{equation*}
has been given. However, since a straightforward generalization of the computation leading to \eqref{eqdomd} shows that
\begin{align}\label{ES}
\E\bigl[D_kF D_kG\bigr]&= \frac12\E\biggl[\Bigl(f\bigl(\X^{(k)}\bigr)-f\bigl(\X\bigr)\Bigr) \Bigl(g\bigl(\X^{(k)}\bigr)-g\bigl(\X\bigr)\Bigr)\biggr],\quad k\in\N,
\end{align}
 these two definitions actually coincide. Moreover, in \cite[Corollary 3.5]{DH} the important \textit{Poincar\'{e} inequality} 
\begin{equation}\label{poincare}
\Var(F)\leq \mathcal{E}(F,F)=\sum_{k=1}^\infty\E\bigl[(D_kF)^2\bigr], \quad F\in\dom(D),
\end{equation}
has been provided, which, thanks to \eqref{ES}, may thus now also be written as 
\begin{equation*}
\Var(F)\leq \frac{1}{2}\sum_{k=1}^\infty\E\biggl[\Bigl(f\bigl(\X^{(k)}\bigr)-f\bigl(\X\bigr)\Bigr)^2 \biggr], \quad F\in\dom(D).
\end{equation*}
As has been remarked in \cite{DH}, this inequality is an infinite version of the celebrated \textit{Efron-Stein inequality} (see e.g. \cite{EfSt,Steele,BLM}). It further trivially continues to hold for all $F\in L^2_\X\setminus \dom(D)$ since, in this case, the right hand side equals $+\infty$ by Proposition \ref{domD}. The inequality \eqref{poincare} has actually implicitly been established in the proof of Proposition \ref{domD}, since therein we have seen that, for $F\in\dom(D)$ with the infinite Hoeffding decomposition \eqref{genhd2}, one has 
\begin{align}\label{devpoincare}
\sum_{k=1}^\infty\E\bigl[(D_kF)^2\bigr]=\sum_{M\in\Pot_{fin}(\N)}|M|\,\E\bigl[F_M^2\bigr] 
\end{align}
which is of course not smaller than 
\[\sum_{\emptyset\not=M\in\Pot_{fin}(\N)}\,\E\bigl[F_M^2\bigr]=\Var(F).\]
Observe that our derivation also indicates how far the inequality is from being an equality. Note further that, for $F\in\bigoplus_{p=0}^m\mathcal{H}_p$, from \eqref{poincare} and \eqref{devpoincare} we obtain the chain of inequalities 
\begin{align}\label{poincare2}
\Var(F)\leq\sum_{k=1}^\infty\E\bigl[(D_kF)^2\bigr]\leq m\Var(F).
\end{align}
Thus, for $F$ living in the sum of the first $m$ chaoses associated to $\X$, the variance and its upper bound via the Poincar\'{e} inequality coincide up to a constant depending on $m$.
\subsection{Covariance formulae}\label{covid}
In order to adapt Stein's method of distributional approximation to the abstract framework, it is important that several useful covariance identities can be established, leading in turn to fruitful new \textit{Stein identities} (see e.g. \cite{CGS}). In this subsection we provide three such formulae.

The first one, the Malliavin type covariance formula, has implicitly been used in \cite{DH} and a very similar formula has further been provided in \cite[Lemma 2.5]{Duer}. We state it here for the sake of later reference and also provide its short proof.
\begin{prop}\label{mallcovid}
For all $F\in L^2_\X$ and $G\in\dom(D)$ one has
\begin{equation*}
\Cov(F,G)=-\E\Bigl[\bigl\langle D L^{-1}\bigl(F-\E[F]\bigr),DG\bigr\rangle_{L^2(\kappa)}\Bigr].
\end{equation*}
\end{prop}

\begin{proof}
Since $L^{-1}(F-\E[F])\in\dom(L)\subseteq\dom(D)$, using $\delta D=-L$ and \eqref{intpartsM} we have
\begin{align*}
\Cov(F,G)&=\E\bigl[(F-\E[F]) G\bigr]=\E\bigl[LL^{-1}(F-\E[F]) G\bigr]\\
&=-\E\bigl[\delta DL^{-1}((F-\E[F]) G\bigr]
=-\E\Bigl[\bigl\langle D L^{-1}\bigl(F-\E[F]\bigr),DG\bigr\rangle_{L^2(\kappa)}\Bigr].
\end{align*}
\end{proof}

The first equality in the next covariance formula, which is of the Clark-Ocone type, is Theorem 3.6 in \cite{DH}.
\begin{prop}\label{clarkcovid}
For all $F,G\in \dom(D)$ one has
\begin{align*}
\Cov(F,G)&=\E\Biggl[\sum_{k=1}^\infty D_k\E[F|\F_k]\, D_kG  \Biggr]=\sum_{k=1}^\infty \E\Bigl[ D_k\E[F|\F_k]\, D_kG  \Bigr] \\
&=\sum_{k=1}^\infty \E\Bigl[ \E[D_kF|\F_k]\,\E[D_kG|\F_k]  \Bigr]
\end{align*}
\end{prop}

Note that the second equality in Proposition \ref{clarkcovid} holds true as $F,G\in\dom(D)$ and, thus,  
\begin{align}\label{clarkex}
\sum_{k=1}^\infty \E\Bigl[\bigl( D_k\E[F|\F_k]\bigr)^2\Bigr]&=\sum_{k=1}^\infty \E\Bigl[\bigl(\E[ D_kF|\F_k]\bigr)^2\Bigr]
\leq\sum_{k=1}^\infty \E\Bigl[\E\bigl[( D_kF)^2|\F_k\bigr]\Bigr]\notag\\
&=\sum_{k=1}^\infty \E\bigl[( D_kF)^2\bigr]<\infty,
\end{align}
where the final inequality is by Proposition \ref{domD}. Hence, by the Cauchy-Schwarz inequality, it is justified to interchange the expectation and the infinite sum. Moreover, the third equality follows from the fact that $D_k\E[F|\F_k]=\E[D_kF|\F_k]$ for any $k\in\N$ and any $F\in L^2_\X$. In particular, for any $F\in \dom(D)$, one has the identity
\begin{align}\label{varid} 
\Var(F)&=\sum_{k=1}^\infty \E\Bigl[ \bigl(\E[D_kF|\F_k]\bigr)^2\Bigr]
\end{align}
which, by an application of the conditional Jensen inequality, again yields \eqref{poincare}.

Finally, we present a novel covariance formula that makes use of the carr\'{e}-du-champ integration by parts formula in Proposition \ref{intpartsgamma0}. It can be considered an infinite analogue of the covariance formula from \cite[Lemma 2.5]{D23} in the context of non-linear exchangeable pairs.
\begin{prop}\label{cdccovid}
For all $F\in L^2_\X$ and $G\in \dom(D)$ one has
\begin{equation*}
\Cov(F,G)=\E\biggl[\Gamma_0\Bigl(-L^{-1}\bigl(F-\E[F]\bigr),G\Bigr)\biggr].
\end{equation*}
\end{prop}

\begin{proof}
Since $G\in \dom(D)$ and $L^{-1}(F-\E[F])\in \dom(L)$, by Proposition \ref{intpartsgamma0} we have
\begin{align*}
\Cov(F,G)&=\E\bigl[(F-\E[F]) G\bigr]=\E\bigl[LL^{-1}(F-\E[F]) G\bigr]\\
&=-\E\Bigl[\Gamma_0\bigl(L^{-1}\bigl(F-\E[F]\bigr),G\bigr)\Bigr]=\E\Bigl[\Gamma_0\bigl(-L^{-1}\bigl(F-\E[F]\bigr),G\bigr)\Bigr].
\end{align*}
\end{proof}

\begin{remark}
At the end of this section on Malliavin calculus and infinite Hoeffding decompositions we make a final comment regarding the generality of the theory initiated in \cite{DH,Duer} and extended in the present work. Suppose that $I$ is in fact an uncountable index set and that, for each $i\in I$, $(E_i,\B_i,\mu_i)$ is a probability space. Denote by 
$(E,\B,\mu):=(\prod_{i\in I} E_i,\bigotimes_{i\in I}\B_i,\bigotimes_{i\in I}\mu_i)$ the corresponding product probability space and by $Y_L:(E,\B)\rightarrow (\prod_{i\in L}E_i,\bigotimes_{i\in L}\B_i)$ the canonical projections, $L\subseteq I$. From the well-known fact that any $B\in\B$ is already contained in 
$\sigma(Y_{J_B})$ for some countable $J_B\subseteq I$ and since the Borel-$\sigma$-field $\B(\R)$ on $\R$ is countably generated, 
one can infer that any measurable $f:(E,\B)\rightarrow(\R,\B(\R))$ is actually measurable with respect to $\sigma(Y_{K})$ for some countable $K\subseteq I$. Hence, as long as only countably many such functionals $f$ are considered, the theory presented here in principle covers the most general case of functionals on products of arbitrarily many probability spaces.
\end{remark}

\section{An infinite quantitative de Jong theorem}\label{dejong}
In this section we prove an infinite version of the quantitative de Jong CLTs that have recently been provided in \cite{DP17,Doe23b}. We stress that already the statement of the following result relies on the concept of infinite Hoeffding decompositions as outlined in Subsection \ref{Hoeffding} and could hence not be given in the classical setup of functionals of finitely many independent random variables. We make use of the framework and the notation introduced in Section \ref{malliavin}.

For $k\in\N$, the \textit{influence} of the variable $X_k$ on $F\in L^2_\X$ is customarily defined as $\Inf_k(F):=\E[\Var(F|\G_k)]$. From the results in Subsection \ref{mallop} it follows that, in terms of the decomposition \eqref{genhd2}, one has the alternative formula
\begin{align*}
\Inf_k(F)=\sum_{\substack{M\subseteq\N:\\ |M|<\infty,\, k\in M}} \E[F_M^2]=\sum_{\substack{M\subseteq\N:\\ |M|<\infty,\, k\in M}} \Var(F_M)=\E\bigl[(D_kF)^2\bigr].
\end{align*}
We further define the \textit{maximal influence} of any single variable $X_k$, $k\in\N$, on $F$ as 
\begin{align*}
\rho^2(F):=\sup_{k\in\N}\Inf_k(F)=\sup_{k\in\N}\sum_{\substack{M\subseteq\N:\\ |M|<\infty,\, k\in M}} \E[F_M^2]
\end{align*}
and let $\rho(F):=\sqrt{\rho^2(F)}$.
In general, one speaks of a \textit{low influence functional}, if $\rho^2(F)$ is small compraed to its variance, see e.g. \cite{MOO10}. As has been demonstrated in \cite{MOO10, NPR2}, the quantity $\rho^2(F)$ plays a fundamental role for the universality of multilinear polynomial forms in independent random variables. Moreover, see again \cite{MOO10}, many important current problems in social choice theory and theoretical computer science are only stated for the low influence case in order to exclude pathological and therefore unimportant counterexamples. 

From now on we additionally suppose that, for some $p\in\N$, $F\in\mathcal{H}_{p,\X}$ is a completely degenerate $U$-statistic of order $p$ based on $\X$ as defined in Subsection \ref{Hoeffding}, that is, one has $F_M=0$ $\P$-a.s. for all finite subsets $M$ of $\N$ with $|M|\not=p$.

\begin{theorem}[Infinite quantitative de Jong CLT]\label{infdejong}
With the above notation fixed, suppose that, for some $p\in\N$, $F\in L^4(\P)$ is a completely degenerate $U$-statistic of order $p$ based on $\X$ such that $\E[F^2]=\Var(F)=1$. Then, for $Z\sim N(0,1)$ we have the bounds 
\begin{align}
d_\W(F,Z)&\leq \Biggl(\sqrt{\frac{2}{\pi}}+\frac{4}{3}\Biggr)\sqrt{\babs{\E[F^4]-3}}+\sqrt{\kappa_p}\Biggl(\sqrt{\frac{2}{\pi}}+ 
\frac{2\sqrt{2}}{\sqrt{3}}\Biggr)\rho(F),\quad\text{and}  \label{djwass}\\
d_\K(F,Z)&\leq  11.9\sqrt{\babs{\E[F^4]-3}} +\bigl(3.5+10.8\sqrt{\kappa_p}\bigr)\rho(F).\label{djkol}
\end{align}
Here, $\kappa_p\in(0,\infty)$ is a combinatorial constant that only depends on $p$.
\end{theorem}

In particular, Theorem \ref{infdejong} implies the following infinite generalization of a classical CLT by P. de Jong \cite[Theorem 1]{deJo90}. Multivariate and functional extensions of the result in \cite{deJo90} have been provided in \cite{DP17} and \cite{D19}, respectively.

\begin{cor}\label{dejocor}
Fix $p\in\N$ and suppose that, for each $n\in\N$, $F_n$ is a normalized, degenerate $U$-statistic of order $p$ based on an independent, finite or countably infinite sequence $\X_n$ that is defined on some probability space $(\Om_n,\F_n,\P_n)$. If $\lim_{n\to\infty}\E[F_n^4]=3$ and $\lim_{n\to\infty}\rho^2(F_n)=0$, then $F_n$ converges in distribution to $Z\sim N(0,1)$ as $n\to\infty$. 
\end{cor}

\begin{remark}\label{dejorem}
\begin{enumerate}[(a)]
\item The bounds in Theorem \ref{infdejong} are direct generalizations, to the setting of an infinite underlying sequence $\X$, of previous bounds by G. Peccati and the author \cite[Theorem 1.3]{DP17} for the Wasserstein distance and by the author \cite[Theorem 2.1]{Doe23b} for the Kolmogorov distance (see the bounds \eqref{djw} and \eqref{djk} below).
\item The constant $\kappa_p$ appearing in the above bounds stems from the article \cite{DP17}, where it is shown that one may choose $\kappa_p=2+C_p$ and the finite combinatorial constant $C_p$ is (rather implicitly) defined in display (4.5) of \cite{DP17}.  
\item Theorem \ref{infdejong} is the counterpart to the quantitative fourth moment theorems on Gaussian \cite{NouPecbook} and Poisson spaces \cite{DP18a, DP18c, DVZ18} and for Rademacher chaos \cite{DK19}. In fact, it is a generalization of \cite[Theorem 1.1]{DK19}. Contrary to the Gaussian and Poisson situations, though, it is in general not possible to remove the quantity $\rho^2(F)$ from the bound as has been shown in \cite[Theorem 1.6]{DK19}.  
\end{enumerate}
\end{remark}

\subsection{Proof of Theorem \ref{infdejong}}\label{djproof}
We make use of the following lemma, the first part of which is certainly known. For the second part, we have not been able to find a suitable reference, though. It is thus of independent interest.

\begin{lemma}\label{dislemma}
Let $Y,W,W_n$, $n\in\N$, be real-valued random variables. 
\begin{enumerate}[{\normalfont(a)}]
\item If $W,W_n\in L^1(\P)$, $n\in\N$, for the same probability space $(\Om,\F,\P)$ and $(W_n)_{n\in\N}$ converges to $W$ in $L^1(\P)$, then $d_\W(W,Y)\leq\liminf_{n\to\infty} d_\W(W_n,Y)$.
\item If $(W_n)_{n\in\N}$ converges in distribution to $W$, then $d_\K(W,Y)\leq\liminf_{n\to\infty} d_\K(W_n,Y)$.
\end{enumerate}
\end{lemma}

\begin{remark}
\begin{enumerate}[(a)]
\item Note that the statement of Lemma \ref{dislemma} is reminiscent of the well-known behaviour of norms under weak and weak-$\ast$ convergence in a functional analysis context. It is unclear to the author if this observation may be turned into a rigorous proof though.
\item Note that part (b) of Lemma \ref{dislemma} does not simply follow from the triangle inequality $d_\K(W,Y)\leq d_\K(W,W_n)+d_\K(W_n,Y)$ by letting $n\to\infty$, since, in general, $d_\K(W,W_n)\rightarrow0$ as $n\to\infty$ only holds, when the distribution function of $W$ is continuous. This, however, is not guaranteed in our application of Lemma \ref{dislemma} to the proof of Theorem \ref{infdejong}.
\end{enumerate}
\end{remark}

\begin{proof}[Proof of Lemma \ref{dislemma}]
To prove (a) fix $h\in\Lip(1)$. Then, 
\begin{align*}
&\babs{\E[h(W)]-\E[h(Y)]}\leq \babs{\E[h(W)]-\E[h(W_n)]}+\babs{\E[h(W_n)]-\E[h(Y)]}\\
&\leq  \babs{\E[h(W)]-\E[h(W_n)]}+d_\W(W_n,Y)\leq \E\babs{W-W_n}+d_\W(W_n,W).
\end{align*}
Hence, as the left hand side does not depend on $n$ it follows from $L^1$-convergence that 
\[\babs{\E[h(W)]-\E[h(Y)]}\leq\liminf_{n\to\infty}\Bigl(\E\babs{W-W_n}+d_\W(W_n,Y)\Bigr)=\liminf_{n\to\infty}d_\W(W_n,Y).\]
Taking the supremum over all $h\in\Lip(1)$ yields (a). 

To  prove (b) fix $z\in\R$. If $z$ is a continuity point of the distribution function $F_W$ of $W$, then, by convergence in distribution, we have
\begin{align*}
\babs{\P(W\leq z)-\P(Y\leq z)}=\lim_{n\to\infty}\babs{\P(W_n\leq z)-\P(Y\leq z)}\leq \liminf_{n\to\infty} d_\K(W_n,Y).
\end{align*}
On the other hand, if $F_W$ is discontinuous at $z$, then, for given $\eps>0$ we may choose a continuity point $z_\eps$ of $F_W$ such that $z<z_\eps<z+\eps$. Then, again by convergence in distribution, 
we have  
\begin{align*}
&\P(W\leq z)-\P(Y\leq z)\leq \P(W\leq z_\eps)-\P(Y\leq z_\eps) +  \P(z< Y\leq z_\eps) \notag\\
& \leq \babs{\P(W\leq z_\eps)-\P(Y\leq z_\eps)} +F_Y(z+\eps)-F_Y(z)\notag\\
&=\lim_{n\to\infty} \babs{\P(W_n\leq z_\eps)-\P(Y\leq z_\eps)}+F_Y(z+\eps)-F_Y(z)\notag\\
&\leq \liminf_{n\to\infty} d_\K(W_n,Y)+F_Y(z+\eps)-F_Y(z)
\end{align*}
and, since $\eps>0$ was arbitrary and the distribution function $F_Y$ of $Y$ is right-continuous at $z$, letting $\eps\downarrow0$, we obtain that 
\begin{align}\label{kol2}
\P(W\leq z)-\P(Y\leq z)\leq\liminf_{n\to\infty} d_\K(W_n,Y). 
\end{align} 
Moreover, again by convergence in distribution and since $(-\infty,z]$ is a closed subset of $\R$, the Portmanteau theorem implies that
\begin{align}\label{kol1}
\P(Y\leq z)-\P(W\leq z)&\leq \P(Y\leq z)-\limsup_{n\to\infty}\P(W_n\leq z)\notag\\
&= \liminf_{n\to\infty}\bigl(\P(Y\leq z)-\P(W_n\leq z)\bigr)\leq \liminf_{n\to\infty}d_\K(W_n,Y).
\end{align}
Since $z\in\R$ was arbitrary, \eqref{kol2} and \eqref{kol1} imply (b).
\end{proof}

\begin{proof}[Proof of Theorem \ref{infdejong}]
Define the $\mathbb{F}$-martingale $(F_n)_{n\in\N}$, where 
\[F_n:=\E\bigl[F\,\big|\,\F_n\bigr]=\sum_{M\subseteq[n]: |M|=p} F_M,\quad n\in\N.\]
In particular, for $n\in\N$, $F_n\in L^4_\X$ is a degenerate $U$-statistic of order $p$ based on $X_1,\dotsc,X_n$ and $(F_n)_{n\in\N}$ converges to $F$ $\P$-a.s. and in $L^4(\P)$ by the $L^4$-convergence theorem for martingales. Further, for $n\in\N$, let $\sigma_n^2:=\Var(F_n)$ and note that $\lim_{n\to\infty}\sigma_n^2=1$ and, hence, we may w.l.o.g. assume that $\sigma_n^2>0$ and define $G_n:=\sigma_n^{-1} F_n$ for any $n\in\N$.  Then, we have that also $(G_n)_{n\in\N}$ converges in $L^4(\P)$ and, hence, in particular in $L^1(\P)$ and in distribution to $F$. Since $G_n$ is normalized, from \cite[Theorem 1.3]{DP17} and \cite[Theorem 2.1]{Doe23b}, for any $n\in\N$, we have the bounds
 \begin{align}
d_\W(G_n,Z)&\leq \Biggl(\sqrt{\frac{2}{\pi}}+\frac{4}{3}\Biggr)\sqrt{\babs{\E[G_n^4]-3}}+\sqrt{\kappa_p}\Biggl(\sqrt{\frac{2}{\pi}}+ \frac{2\sqrt{2}}{\sqrt{3}}\Biggr)\rho(G_n)\quad\text{and}  \label{djw}\\
d_\K(G_n,Z)&\leq  11.9\sqrt{\babs{\E[G_n^4]-3}} +\bigl(3.5+10.8\sqrt{\kappa_p}\bigr)\rho(G_n).\label{djk}
\end{align} 
Now, noting that, for each $k\in\N$, as $n\to\infty$,
\begin{align*}
\Inf_k(F_n)&=\sum_{M\subseteq[n]:k\in M}\E[F_M^2]\Big\uparrow \sum_{M\subseteq\N:k\in M}\E[F_M^2]=\Inf_k(F)
\end{align*}
we obtain that also $\rho(F_n^2)$ is increasing in $n$ and, thus, 
\begin{align}\label{dj4}
\lim_{n\to\infty}\rho^2(F_n)=\sup_{n\in\N}\sup_{k\in\N}\Inf_k(F_n)=\sup_{k\in\N}\sup_{n\in\N}\Inf_k(F_n)=\sup_{k\in\N}\Inf_k(F)=\rho^2(F).
\end{align}
Furthermore, as $\rho(G_n)^2=\sigma_n^{-2}\rho^2(F_n)$, $n\in\N$, from \eqref{dj4} we have that $\lim_{n\to\infty}\rho^2(G_n)=\rho^2(F)$. Thus, from Lemma \ref{dislemma}, \eqref{djw}, \eqref{djk} and the $L^4$-convergence of $(G_n)_{n\in\N}$ to $F$ we conclude that 
\begin{align*}
d_\W(F,Z)&\leq\liminf_{n\to\infty}d_\W(G_n,Z)\leq \Biggl(\sqrt{\frac{2}{\pi}}+\frac{4}{3}\Biggr)\sqrt{\babs{\E[F^4]-3}}+\sqrt{\kappa_p}\Biggl(\sqrt{\frac{2}{\pi}}+ 
\frac{2\sqrt{2}}{\sqrt{3}}\Biggr)\rho(F)
\end{align*} 
and 
\begin{align*}
d_\K(F,Z)&\leq\liminf_{n\to\infty}d_\K(G_n,Z)\leq 11.9\sqrt{\babs{\E[F^4]-3}} +\bigl(3.5+10.8\sqrt{\kappa_p}\bigr)\rho(F),
\end{align*} 
completing the proof of Theorem \ref{infdejong}.
\end{proof}

\section{A probabilistic proof of Proposition \ref{infhoeff}} \label{appendix}
In this section we give a self-contained proof of Proposition \ref{infhoeff} on infinite Hoeffding decompositions. We first review the concepts of unconditional convergence and summability in Banach spaces.

\subsection{Unconditional convergence and summability in Banach spaces } 
Let $(B,\norm{\cdot})$ be a normed space, let $\emptyset\not=I$ be an index set and let $(x_i)_{i\in I}$ be a family of vectors in $B$. 

The (symbolic) series $\sum_{i\in I} x_i$ is said to \textit{converge unconditionally} to the 
vector $x\in B$, if the following two conditions hold:
\begin{itemize}
 \item The set $I^*:=\{i\in I\,:\, x_i\not=0\}$ is at most countably infinite.
 \item If $\abs{I^*}=\infty$ and $I^*=\{i_n\, :\, n\in\N\}$ is an enumeration of $I^*$, then the series $\sum_{n=1}^\infty x_{i_n}$ converges to $x$, i.e.
 \[\lim_{m\to\infty}\Bigl\|\sum_{n=1}^m x_{i_n} -x\Bigr\|=0\,.\]
\end{itemize}
On the other hand, the family $(x_i)_{i\in I}$ is called \textit{summable}, if there is an $s\in B$ with the following property: For each $\epsilon>0$ there is a finite set $I_\eps\subseteq I$ such that 
for all finite $J$ with $I_\eps\subseteq J\subseteq I$ we have 
\[\Bigl\|\sum_{i\in J} x_i -s\Bigr\|<\eps\,.\]
It is not difficult to see that such an $s$ is necessarily unique. Moreover, it is a standard result in functional analysis that summability and unconditional convergence of series in Banach spaces are equivalent and that, with the above notation, one has $x=s$.

In finite-dimensional spaces, these conditions are further equivalent to absolute convergence of the series, i.e. to the condition that $\sum_{n\in\N}\norm{x_{i_n}}<\infty$. 

For infinite-dimensional Banach spaces $B$, however, according to the famous \textit{Dvoretzky-Rogers theorem} \cite{DvoRo}, there always exists a series $\sum_{i\in I} x_i$ that is unconditionally but not absolutely convergent.

The following result about the unconditional convergence of families of centered and uncorrelated random variables in $L^2(\P)$ is certainly known. As we have not found a suitable reference, though, we include its statement and also provide a complete proof.

\begin{prop}\label{serieslemma}
Let $\emptyset\not=I$ be 
an index set and suppose that $(X_i)_{i\in I}$ is a family of square-integrable, centered and pairwise uncorrelated random variables in $L^2(\P)$ for some probability space $(\Omega,\A,\P)$. Assume that the family $(\Var(X_i))_{i\in I}$ of nonnegative real numbers satisfies
\[\sup\Bigl\{\sum_{i\in J}\Var(X_i)\,:\,J\subseteq I\text{ finite }\Bigr\}<\infty. \]
Then, there is a random variable $S\in L^2(\P)$ such that the series $\sum_{i\in I} X_i$ converges unconditionally to $S$ in the $L^2(\P)$-sense.
\end{prop}

\begin{proof}
It is not difficult to see that the assumption implies that the set $I^*:=\{i\in I\,:\, \Var(X_i)\not=0\}$ is countable and we may assume without loss of generality that it is countably infinite, since in the finite case the result is immediate.
Let us first fix an enumeration $(i_n)_{n\in\N}$ of $I^*$. Then, we obviously have 
\[\sum_{n=1}^\infty \Var(X_{i_n})=\sup\Bigl\{\sum_{i\in J}\Var(X_i)\,:\,J\subseteq I\text{ finite }\Bigr\}<\infty.\]
Therefore, using orthogonality, for integers $1\leq m<n$ it follows that 
\begin{align*}
 \Bigl\|\sum_{k=m+1}^n X_{i_k}\Bigr\|_2^2&=\sum_{k,l=m+1}^n\E\bigl[X_{i_k}X_{i_l}\bigr]=\sum_{k=m+1}^n\Var\bigl(X_{i_k}\bigr)
 \longrightarrow0\,,\quad\text{as }n,m\to\infty.
\end{align*}
Hence, the sequence $(S_n)_{n\in\N}$ with $S_n=\sum_{k=1}^n X_{i_k}$ is Cauchy in $L^2(\P)$ and converges by completeness to some $S\in L^2(\P)$. 
It remains to make sure that this random variable $S$ does not depend on the particular enumeration chosen. Thus, let $\pi:\N\rightarrow\N$ be any bijection. Then, as above, one obtains that also the series 
$\sum_{k=1}^\infty X_{i_{\pi(k)}}$ converges to some $S_\pi\in L^2(\P)$. We show that $S=S_\pi$ $\P$-a.s. Let $\eps>0$ be given and choose $n_0=n_0(\eps)\in\N$ such that 
\[\sum_{k=n_0+1}^\infty \Var\bigl(X_{i_k}\bigr)<\eps^2\,.\]
Furthermore, choose $m_0=m_0(\eps)\in\N$ such that $[n_0]\subseteq\pi([m_0])$. Then, by the triangle inequality, we obtain 
\begin{align*}
 \Enorm{S-S_\pi}&\leq\BEnorm{S-\sum_{k=1}^{n_0}X_{i_k}}+\BEnorm{\sum_{k=1}^{n_0}X_{i_k}-\sum_{k=1}^{m_0}X_{i_{\pi(k)}}}+\BEnorm{\sum_{k=1}^{m_0}X_{i_{\pi(k)}}-S_\pi}\\
 &\leq \Bigl(\sum_{k=n_0+1}^\infty\Var\bigl(X_{i_k}\bigr)\Bigr)^{1/2}+\Bigl(\sum_{l\in\pi([m_0])\setminus[n_0]}\Var\bigl(X_{i_{l}}\bigr)\Bigr)^{1/2}\\
&\; +\Bigl(\sum_{l\in\N\setminus\pi([m_0])}\Var\bigl(X_{i_l}\bigr)\Bigr)^{1/2}\leq 3 \Bigl(\sum_{k=n_0+1}^\infty\Var\bigl(X_{i_k}\bigr)\Bigr)^{1/2}<3\epsilon\,.
\end{align*}
Since $\eps>0$ was arbitrary, this implies that $S=S_\pi$ $\P$-a.s.and $\sum_{i\in I} X_i$ converges unconditionally to $S$.\\
\end{proof}

\subsection{Proof of Proposition \ref{infhoeff} }
Recall that $F_M=F_{\max(M),M}=F_{n,M}$ for all finite $M\subseteq\N$ and all $n\geq\max(M)$. 
 We first prove \eqref{genhd2}. Due to \eqref{cons3}, for $n\in\N$ we have 
 \begin{align}\label{hd1}
  F_n&=\sum_{\substack{M\subseteq\N:\\ \max(M)\leq n}}F_{M}.
\end{align}
Hence, using orthogonality and \eqref{l2bound2} we obtain 
\begin{align}\label{hd2}
 \sum_{\substack{M\subseteq\N:\\ \max(M)\leq n}}\Var\bigl(F_{M}\bigr)&
 =\Var(F_n)\leq \Var(F)\,.
\end{align}
As the right hand side of the inequality \eqref{hd2} does not depend on $n$ we conclude that 
\begin{align}\label{hd3}
 \sum_{\substack{M\subseteq\N:\\ \abs{M}<\infty}}\Var\bigl(F_{M}\bigr)&=\lim_{n\to\infty}\sum_{\substack{M\subseteq\N:\\ \max(M)\leq n}}\Var\bigl(F_{M}\bigr)=\lim_{n\to\infty}\Var(F_n)=\Var(F)<+\infty\,,
\end{align}
where we have used the $L^2(\P)$ convergence of $(F_n)_{n\in\N}$ to $F$ to obtain the last identity. As the summands are pairwise orthogonal and $L^2(\P)$ is a Banach space, by Proposition \ref{serieslemma} this already implies that the first infinite series $\sum_{\substack{M\subseteq\N: \abs{M}<\infty}}F_{M} $ on the right hand side of \eqref{genhd2} converges unconditionally in the $L^2(\P)$-sense to some limit $\tilde{F}\in L^2_\X$. We need to make sure that $\tilde{F}=F$ holds $\P$-a.s. As $(F_n)_{n\in\N}$ converges to $F$ in $L^2(\P)$ this will follow if we can show that also $\tilde{F}$ is the $L^2(\P)$-limit of $(F_n)_{n\in\N}$. Let $(M_k)_{k\in\N}$ be an enumeration of the finite subsets of $\N$. Then, by unconditional convergence, we have 
\[\tilde{F}=\sum_{k=1}^\infty F_{M_k}\]
in the $L^2(\P)$-sense. Thus, given $\eps>0$ we may choose $k_0\in\N$ such that 
\[\Biggl(\sum_{k=k_0+1}^\infty\E\bigl[F_{M_k}^2\bigr]\Biggr)^{1/2}=\BEnorm{\sum_{k=k_0+1}^\infty F_{M_k}}=\BEnorm{\tilde{F}-\sum_{k=1}^{k_0} F_{M_k}} <\eps.\]
Then, for $n\geq n_0$, where $n_0\in\N$ is chosen in such a way that $M_1,\dotsc,M_{k_0}\subseteq[n_0]$ it holds that
\begin{align*}
 \Enorm{\tilde{F}-F_n}&=\Biggl(\sum_{\substack{k\in\N:\\ M_k\not\subseteq [n]}}\E\bigl[F_{M_k}^2\bigr]\Biggr)^{1/2}\leq \Biggl(\sum_{k=k_0+1}^\infty\E\bigl[F_{M_k}^2\bigr]\Biggr)^{1/2}
 <\eps,
\end{align*}
as desired. Thus, $F=\tilde{F}$ $\P$-a.s.
This proves the first identity in \eqref{genhd2}. By the same argument with obvious adaptations one shows that \eqref{genhd1} is true for each $p\in\N_0$. 
It remains to prove the second identity in \eqref{genhd2}. Since 
\begin{align*}
 &\sup_{m\in\N}\sum_{p=0}^m\Var\bigl(F^{(p)}\bigr)= \sup_{m\in\N}\lim_{n\to\infty}\sum_{p=0}^m\Var\bigl(F^{(p)}_n\bigr)=\sup_{m\in\N}\sup_{n\in\N}\sum_{p=0}^m\Var\bigl(F^{(p)}_n\bigr)\\
 &=\sup_{n\in\N}\sup_{m\in\N}\sum_{p=0}^m\Var\bigl(F^{(p)}_n\bigr)=\sup_{n\in\N}\sum_{p=0}^n\Var\bigl(F^{(p)}_n\bigr)=\sup_{n\in\N}\Var(F_n)=\Var(F)<\infty,
\end{align*}
it follows from \eqref{orthrel} and Proposition \ref{serieslemma} again that the series $\sum_{p=0}^\infty F^{(p)}$ converges unconditionally to some $T\in L^2_\X$ in the $L^2(\P)$-sense. Given $\eps>0$ we can hence find an $m_0\in\N$ such that for all $m\geq m_0$ 
\[\Biggl(\sum_{p=m+1}^\infty\E\Bigl[\bigl(F^{(p)} \bigl)^2\Bigr]\Biggr)^{1/2}=\BEnorm{\sum_{p=m+1}^\infty F^{(p)}}=\BEnorm{T-\sum_{p=0}^{m} F^{(p)}}<\eps.\]
Then, for given $\eps>0$ with $(M_k)_{k\in\N}$ and $k_0$ as above we choose $m_1\geq m_0$ large enough to ensure that $|M_{k}|\leq m_1$ for $k=1,\dotsc,k_0$. Then, we have 
\begin{align*}
 \Enorm{T-F}&\leq\BEnorm{T-\sum_{p=0}^{m_1} F^{(p)}}+\BEnorm{\sum_{p=0}^{m_1} F^{(p)}-\sum_{k=1}^{k_0}F_{M_k}}+\BEnorm{\sum_{k=k_0+1}^{\infty}F_{M_k}}\\
 &\leq \BEnorm{\sum_{p=m_1+1}^\infty F^{(p)}}+2\BEnorm{\sum_{k=k_0+1}^{\infty}F_{M_k}}<3\eps,
\end{align*}
where we have used the fact that the convergence in 
\[\sum_{p=0}^{m_1} F^{(p)}=\sum_{\substack{M\subseteq\N:\\ |M|\leq m_1}} F_M\]
is again unconditional, which follows from the unconditional convergence in \eqref{genhd1} and the fact that summability is obviously preserved under addition of finitely many series.
Thus, $T=F$ $\P$-a.s. which finishes the proof of Proposition \ref{infhoeff}.

\normalem
\bibliography{malliavin_product}{}
\bibliographystyle{alpha}
\end{document}